\documentclass[onefignum,onetabnum]{siamart220329}

\usepackage{lipsum}
\usepackage{amsfonts}
\usepackage{graphicx}
\usepackage{epstopdf}
\usepackage{algorithmic}
\usepackage{mathrsfs}
\usepackage{multirow}
\usepackage{matlab-prettifier}
\usepackage{makecell}
\usepackage{soul}
\ifpdf
  \DeclareGraphicsExtensions{.eps,.pdf,.png,.jpg}
\else
  \DeclareGraphicsExtensions{.eps}
\fi

\newcommand{\mat}[1]{\lstinline[style=Matlab-editor]{#1}}
\newcommand\ChangeRT[1]{\noalign{\hrule height #1}}

\newsiamremark{remark}{Remark}
\newsiamremark{hypothesis}{Hypothesis}
\crefname{hypothesis}{Hypothesis}{Hypotheses}
\newsiamthm{claim}{Claim}
\newsiamremark{prop}{Proposition}

\headers{}{}

\title{Mixed Precision Iterative Refinement for Linear Inverse Problems}

\author{James G. Nagy\thanks{Department of Mathematics, Emory University, Atlanta, GA 
  (\email{jnagy@emory.edu}), (\email{lonisk@emory.edu}). This work was partially supported by the U.S. National Science Foundation, Grants DMS-2038118 and DMS-2208294.}
\and Lucas Onisk\footnotemark[1]}

\newcommand{\rev}[1]{\textcolor{black}{#1}}

\newcommand{\Rev}[1]{\textcolor{black}{#1}}

\begin{document}

\maketitle

\begin{abstract}
This study investigates the iterative refinement method applied to the solution of linear discrete inverse problems by considering its application to the Tikhonov problem in mixed precision. Previous works on mixed precision iterative refinement methods for the solution of symmetric positive definite linear systems and least-squares problems have shown regularization to be a key requirement when computing low precision factorizations. For problems that are naturally severely ill-posed, we formulate the iterates of iterative refinement in mixed precision as a filtered solution using the preconditioned Landweber method with a Tikhonov-type preconditioner. Through numerical examples simulating various mixed precision choices, we showcase the filtering properties of the method \Rev{and the achievement of comparable working accuracy of discrete inverse problems (i.e., to within a few decimal places in relative error) compared to results computed in double precision as well as another approximate iterative refinement method which we use as a benchmark.}

\end{abstract}



\section{Introduction} \label{intro}

We are interested in computing approximate solutions to linear least-squares problems of the form
\begin{equation} \label{minProb}
    \min_{x\in \mathbb{R}^{n}} \left\|Ax-b\right\|
\end{equation}
in computing environments where mixed precision floating-point arithmetic is available. Here, $A \in \mathbb{R}^{m \times n}$, $m \geq n$ is a matrix whose singular values decay without significant gap and cluster at the origin (i.e. the matrix is ill-conditioned). We denote the Euclidean norm by $\|\cdot\|$. These types of problems are commonly referred to as linear discrete ill-posed problems. They can arise through the discretization of Fredholm integral equations of the first kind; see \cite{H98, EHN96}, but also arise in massive data streaming problems such as the training of the random feature model in machine learning \cite{RB07} or limited angle tomography problems including, for example, those from medical imaging \cite{Chung_inter3, streamInvProb}.

In applications the vector $b$ in \eqref{minProb} often represents measurements that are corrupted by error arising through noise contamination or instrumentation disturbances. From a numerical point of view this error may also include truncation error or approximation error. We denote this error by $e$ so that
\begin{equation*}
    b = \hat{b} + e
\end{equation*}
where $\hat{b}$ is the unknown error-free vector associated with $b$. Ideally, we would like to determine the solution of minimal Euclidean norm, $x^{\dagger}$, of the unavailable least-squares problem 
\begin{equation} \label{trueMinProb}
   x^{\dagger}= \arg\min_{x\in \mathbb{R}^{n}} \left\|Ax-\hat{b}\right\|.
\end{equation}
It's well understood that because $b$ is contaminated by error and the singular values of $A$ cluster at the origin that the solution of \eqref{minProb} of minimal Euclidean norm is usually a poor approximation to $x^{\dagger}$. Regularization is a common strategy that replaces the problem \eqref{minProb} with a nearby problem that is less sensitive to the error in $b$. 

In response to advancements in GPU computing architectures that natively support mixed precision computing, Higham and Pranesh in \cite{Pranesh21} consider the solution of symmetric positive definite (SPD) linear systems and least-squares problems in multiple precisions using iterative refinement (IR). They found that definiteness of SPD matrices could be lost when rounded to a precision with fewer representation bits, which in turn, could result in the failure of a Cholesky factorization when adapting GMRES-based IR strategies; see \cite{CarsonHigham18, CarsonHighamPranesh20}. \rev{It was determined that performing a diagonal shifting scheme could preserve the definiteness of the matrix - which may be viewed as a form of regularization.} Our interests in this work are focused on applying IR to least-squares problems whose matrix $A$ is naturally severely ill-conditioned without any assumed underlying structure, to which Tikhonov regularization is often applied. Other works by Higham and others concerning the use of IR to solve linear systems using mixed precision include \cite{Carson2, Higham1, Khan23, Carson1}.

In Tikhonov regularization, \eqref{minProb} is replaced by the penalized least-squares problem
\begin{equation} \label{TikEqn}
    \min_{x\in \mathbb{R}^{n}} \left\{\left\|Ax-b\right\|^2 + \alpha^2\left\|Lx\right\|^2\right\},
\end{equation}
where $\alpha>0$ is a regularization parameter and $L \in \mathbb{R}^{s \times n}$ is a regularization matrix. We focus on what is termed \emph{standard} Tikhonov, i.e. when $L = I$. When $L$ is chosen so that the null spaces of $A$ and $L$ trivially intersect then the solution of \eqref{TikEqn} may be expressed as
\begin{equation} \label{TikClosedForm}
    x^{(\alpha)} = \left(A^TA + \alpha^2 I\right)^{-1}A^Tb,
\end{equation}
where the superscript $^T$ denotes the transpose operation. The regularization parameter $\alpha$ may be thought of as a `trade-off' parameter which balances the sensitivity of the solution vector $x^{(\alpha)}$ to the error in $b$, as well as the closeness to the desired solution $x^{\dagger}$.

Given the unexpected necessity of regularization in structure leveraging algorithms for SPD linear systems in \cite{Pranesh21}, we investigate the use of IR to solve the Tikhonov problem in mixed precision providing a natural extension to the aformentioned works. To better understand the regularized solutions of ill-posed problems computed using IR we derive a methodology to formulate the iterates applied to the Tikhonov problem as filtered solutions by writing them as a recursive relationship between the iterates of preconditioned Landweber with a Tikhonov-type preconditioner and previous iterates. Additionally, we demonstrate in our numerical results that mixed precision IR on the Tikhonov problem gives \Rev{comparable working accuracy (i.e., to within a few decimal places in relative error) against results computed in double precision as well as what we term the \emph{Approximated Iterative Refinement} method (see Section \ref{secNum4}),} which supports its use in modern applications that natively support mixed precision floating-point arithmetic.

An outline of this paper is as follows. In Section \ref{sec2} we provide necessary background on topics including the Landweber iteration, IR, and filtering methods. Section \ref{sec3} discusses preconditioned Landweber and its connection to the filtering analysis of IR on the Tikhonov problem in mixed precision, which is discussed in Section \ref{sec4}. Experimental results and concluding remarks are given in Sections \ref{sec5} and \ref{sec6}, respectively.

\section{Background} \label{sec2}
This section provides necessary background and notation for representing the regularized solution of a linear discrete ill-posed problem as a filtered solution. The Landweber and IR methods are also relevantly reviewed.

\subsection{Filtered Solution of Linear Discrete Ill-posed Problems}

Many inverse problems whose approximate solution is given by \eqref{minProb} have been well studied using the singular value decomposition (SVD) of $A=U\Sigma V^T$, where $U \in \mathbb{R}^{m \times m}$ and $V \in \mathbb{R}^{n \times n}$ are orthogonal matrices whose columns represent the left and right singular vectors of $A$, respectively. \rev{The diagonal matrix $\Sigma \in \mathbb{R}^{m \times n}$ contains the singular values given by $\sigma_j$ for $j=1,2,\dots,n$ on its main diagonal ordered in non-increasing fashion.} Using the SVD, one can write the least-squares solution of \eqref{minProb} as a linear combination of the right singular vectors with the $j^{th}$ coefficient given by $u_j^Tb/\sigma_j$. Since the singular values of $A$ cluster near numerical zero, the error $e$ contaminating $b$ is propagated into the computed solution for high index values $j$ corresponding to high frequency information \cite{H98, HNO}.

To prevent the propagation of error into the approximate solution of \eqref{minProb}, the conditioning of the problem can be improved by employing filtering methods which can be formulated as a modified inverse solution \cite{HNO, NP03}. A filtered solution is of the form
\begin{equation} \label{filteredSoln}
    x_{filt.} = \sum_{j=1}^n  \phi_j \frac{u_j^Tb}{\sigma_j}v_j
\end{equation}
where an intelligent selection of the filter factors $\phi_j \in \mathbb{R}$ can remove deleterious components of the approximate solution corresponding to high frequency information that is dominated by noise - which can be considered a regularized solution. To achieve an effective solution this corresponds to $\phi_j \approx 1$ for small $j$ and $\phi_j \approx 0$ for large $j$. Using the SVD of $A$, \eqref{TikClosedForm} may be written as a filtered solution with filter factors given by $\phi_j = \sigma_j^2/(\sigma_j^2 + \alpha^2)$ for $j=1,2,\dots,n$.

\subsection{Landweber Iteration}

A regularized solution of \eqref{minProb} can also be obtained by applying an iterative method and terminating the iterations before the error is propagated. The discrete Landweber iteration \cite{Landweber} is one such method whose $k^{th}$ iterate is given by
\begin{equation} \label{landweber}
    x^{(k)} = x^{(k-1)} + \zeta A^T\left(b-Ax^{(k-1)}\right)
\end{equation}
where $\zeta \in \left(0,1/\sigma_1^2\right]$ is a relaxation term with $\sigma_1$ the largest singular value of $A$. Without loss of generality, we choose $\zeta=1$ in this work since $A$ can be scaled appropriately\rev{; see \cite{engl1996regularization} for further discussion on the choice of $\zeta$ for inverse problems.} While robust, this classical variant of Landweber is well known to have slow convergence towards a useful solution which can be understood from viewing its $k^{th}$ solution as a filtering method \cite{HankeHansen93}:
\begin{equation} \label{filterLand}
        x^{(k)} = \sum_{j=1}^{n} \Big(1-(1-\sigma_j^2)^{k} \Big) \frac{u_j^Tb}{\sigma_j}v_j
\end{equation}
where we denote $\phi^{(k)}_j = \left(1-(1-\sigma_j^2)^{k} \right) \in \mathbb{R}$ as the $j^{th}$ filter factor of the $k^{th}$ iterate. 

For a scaled matrix $A$ whose largest singular value is one, only the first filter factor of the $k^{th}$ iteration, $\phi^{(k)}_1$ will be $1$, with subsequent filter factors of the same iteration rapidly decaying to numerical zero. It is only when $k$ grows that early filtering values, i.e., $\phi^{(k)}_j$ for small $j$, cluster near $1$. Because of this behavior, it can take many iterations to fully capture the dominant right singular vectors that constitute a meaningful approximate solution for inverse problems. One common way to overcome this difficulty is to employ a preconditioner; we comment more on this in Section \ref{sec3}.

\subsection{Iterative Refinement on the Tikhonov Problem}

Variants of fixed and mixed precision IR methods are not new - having been investigated consistently since Wilkinson's first programmed version around 1948 \cite{Wilk48}. A historical treatise of the progression of IR for linear problems may be found in \cite{CarsonHigham18, Higham22}. \rev{On a generic linear system, IR computes the $k^{th}$ approximate solution, $x^{(k)}$, in three steps: (i) the $(k-1)^{th}$ residual $r^{(k-1)}$ is computed, (ii) the correction equation $Ah^{(k-1)}=r^{(k-1)}$ is solved, and (iii) the current solution is updated $x^{(k)}=x^{(k-1)}+h^{(k-1)}$.} The process may be repeated as necessary until a termination criterion is met. 
 
When considering the application of IR to the Tikhonov problem given by
\begin{equation} \label{IR}
    (A^TA + \alpha^2I)x^{(k)}=A^Tb
\end{equation}
which was first discussed by Riley in \cite{Riley55}, Golub in \cite{Golub65} noted that the \rev{three step procedure of IR is equivalent to the procedure of} iterated Tikhonov regularization whose $k^{th}$ iterate is given by \rev{
\begin{equation} \label{iterTikReg}
    \begin{split}
    x^{(k)} &= x^{(k-1)} + \left(A^TA+\alpha^2I\right)^{-1}A^T\left(b-Ax^{(k-1)}\right)\\
    &= x^{(k-1)} + \left(A^TA+\alpha^2I\right)^{-1}A^Tr^{(k-1)}
    \end{split}
\end{equation}
which is often used as an iterative technique to approximate the solution of \eqref{minProb}. When applying IR to the Tikhonov problem \eqref{IR} the $k^{th}$ residual, $\hat{r}^{(k)}$, may be written as follows
\begin{equation*}
    \begin{split}
        \hat{r}^{(k)} &= A^Tb - \left(A^TA+\alpha^2I\right)x^{(k)}_{IR}\\
        &= A^T\left(b - Ax^{(k)}_{IR}\right) - \alpha^2x^{(k)}_{IR}
    \end{split}
\end{equation*}
so that the $k^{th}$ iterate may be written recursively as
\begin{equation} \label{IR_Tik_expanded}
    \begin{split}
    x^{(k)}_{IR} &= x^{(k-1)}_{IR} + \left(A^TA + \alpha^2I\right)^{-1}\hat{r}^{(k-1)}\\
    &= x^{(k-1)}_{IR} + \left(A^TA + \alpha^2I\right)^{-1}A^Tr^{(k-1)} - \alpha^2 \left(A^TA + \alpha^2I\right)^{-1}x_{IR}^{(k-1)}.
    \end{split}
\end{equation}
}The algorithm for IR on the Tikhonov problem is provided in Algorithm \ref{alg_IRTik} for completeness.

\begin{algorithm}
\caption{Iterative refinement on the Tikhonov problem}
\label{alg_IRTik}
\algsetup{indent=2em,linenodelimiter=}
\begin{algorithmic}[1]
\STATE{\rev{\text{\bf{Input:}} $A \in \mathbb{R}^{m \times n},\,\, b \in \mathbb{R}^m,\,\, x_{IR}^{(0)} \in \mathbb{R}^n,\,\, \alpha>0$}}
\STATE{\rev{\text{\bf{Output:}} $x_{IR}^{(k+1)} \in \mathbb{R}^n$}}
\FOR{$k = 0,1,2,\dots$}
    \STATE{\rev{$r^{(k)} = b - Ax_{IR}^{(k)}$}}
    \STATE{\rev{$s^{(k)} = A^Tr^{(k)}-\alpha^2x_{IR}^{(k)}$}}
    \STATE{Solve $(A^TA + \alpha^2I)h^{(k)} = s^{(k)}$}
    \STATE{\rev{$x_{IR}^{(k+1)} = x_{IR}^{(k)} + h^{(k)}$}}
\ENDFOR
\end{algorithmic}
\end{algorithm}

\rev{From \eqref{IR_Tik_expanded} we may observe that the term $\left(A^TA + \alpha^2I\right)^{-1}\hat{r}^{(k-1)}$ contains the correction term of iterated Tikhonov \eqref{iterTikReg} whose filter factors at the $k^{th}$ iteration take the form}
\begin{equation}
    \phi_j^{(k)} = 1 - \left(1 - \frac{\sigma_j^2}{\sigma_j^2+\alpha^2}\right)^k,\quad \quad j=1,2,\dots,n
\end{equation}
which are described in \cite{Golub65} and \cite{King79}. \rev{The term $\alpha^2 \left(A^TA + \alpha^2I\right)^{-1}x_{IR}^{(k-1)}$ from \eqref{IR_Tik_expanded} may be interpreted as a correction term that refines the approximation at each iteration to the solution of \eqref{IR}.} Differently than discussed in the aforementioned works, we show that the filter factors of IR on the Tikhonov problem can be written with respect to the $k^{th}$ iterate of iterated Tikhonov regularization \eqref{iterTikReg}. This framework lends itself to the filter factor analysis in mixed precision where instead of the $k^{th}$ iterate of \eqref{iterTikReg} we utilize the $k^{th}$ iterate of preconditioned Landweber with a Tikhonov-type preconditioner, which we show in the next section can be written as a filtered solution.

\section{Landweber with a Tikhonov-type Preconditioner} \label{sec3}

To consider the filtering of mixed precision IR on the Tikhonov problem we first consider the preconditioning of the Landweber iteration. Denote the right-preconditioned least-squares problem \eqref{minProb} by
\begin{equation} \label{preconLS}
    \min_{x\in \mathbb{R}^{n}} \left\|AM^{-1}Mx-b\right\|
\end{equation}
with nonsingular preconditioner $M \in \mathbb{R}^{n \times n}$. By denoting $\hat{A}=AM^{-1}$ and $\hat{x}=Mx$ we may use Landweber \eqref{landweber} with $\zeta=1$ to solve \eqref{preconLS}, which may be simplified to
\begin{equation} \label{preconLand}
        x^{(k+1)} = x^{(k)} + \left(M^TM\right)^{-1}A^T\left(b-Ax^{(k)}\right).
\end{equation}
We note that if $M^TM = \left(A^TA + \alpha^2 I\right)$ then the preconditioned Landweber method is equivalent to iterated Tikhonov regularization \eqref{iterTikReg}. In this specific case, utilizing the SVD of $A$ we can represent $M^TM$ as follows
\begin{equation} \label{MM}
    \begin{split}
        M^TM &= \left(V\Sigma^T U^TU\Sigma V^T + \alpha^2 I\right) \\
            &= V\left(\Sigma^T\Sigma + \alpha^2 I\right)V^T \\
            &= VD^2V^T
    \end{split}
\end{equation}
where $D^2 = \left(\Sigma^2 + \alpha^2I\right)$ and $\Sigma^2 = \Sigma^T\Sigma \in \mathbb{R}^{n \times n}$. Thus, one may consider the preconditioning matrix $M$ of \eqref{preconLS} as $M = \left(\Sigma^2 + \alpha^2I\right)^{1/2}V^T=DV^T$: the product of a diagonal matrix with the right singular matrix of $A$. 

Herein, we utilize the $M$ notation to help simplify our notation throughout this work and to consider the situation that $M^TM$ represents an approximate Tikhonov-type preconditioner to $A^TA+\alpha^2I$. Possible approximations we consider in Section \ref{sec5} include that of (i) $M^TM$ as a low precision SVD approximation and (ii) when the structure of $A$ allows for the approximation by a second matrix $C$ whose structure allows for fast and robust matrix-vector products. The former will be our focus when considering IR on the Tikhonov problem in mixed precision. The latter case can arise naturally, for example, in image deblurring problems; see Section \ref{sec5}.

We clarify that there are various possible preconditioning strategies other than Tikhonov that could be used. It should be understood by our usage of the wording `preconditioned Landweber' in this work that we only consider Tikhonov-type preconditioners. For further discussions and other preconditioner types considered for the solution of linear discrete ill-posed problems see e.g., \cite{Estatico08,Bertero97,Strand74}.

To investigate the filter factors of mixed precision IR on the Tikhonov problem we first describe in Proposition \ref{prop1} the filter factors of preconditioned Landweber with $M^TM = \left(A^TA + \alpha^2 I\right)$. We again emphasize that the main use of the $M$ notation is to delineate between the possible change in structure of the preconditioner $M$ or its representation in a lower precision. Throughout, we will use a capital letter subscript, $_Q$, to denote terms that come from a representation of $Q$, e.g., the SVD of $Q$. Additionally, we utilize subscripts $PL$ and $IR$ on vectors to denote solutions from preconditioned Landweber and IR, respectively. 

\begin{prop}{(Filtered solution of preconditioned Landweber)\\} \label{prop1}
The $k^{th}$ iterative solution of \eqref{preconLand} with preconditioner $M = D_MV_M^T$ as defined above may be written as 
\begin{equation} \label{preconLand_wFF}
    \Rev{
	x^{(k)}_{PL} = \sum_{j=1}^n \psi^{(k)}_j\frac{u_{A,j}^Tb}{\sigma_{A,j}}v_{M,j}
    }
\end{equation}
with the $j^{th}$ filter factor per $k^{th}$ iterative step given by 
\begin{equation} \label{FF_preconLand}
\psi^{(k)}_j= 1 - \left(\frac{\sigma_{M,j}^2+ \alpha^2 - \sigma_{A,j}^2}{\sigma^2_{M,j} + \alpha^2}\right)^{k}, \quad j=1,2,\dots,n
\end{equation}
under the assumption that $V_M^TV_A = I = V_A^TV_M$.
\end{prop}

\begin{proof}
We begin by considering Landweber with $\zeta=1$ applied to our right-preconditioned least-squares problem \eqref{preconLS} which may be represented by
\begin{equation*}
	\hat{x}^{(k+1)}_{PL} = \hat{x}^{(k)}_{PL} + \hat{A}^T\left(b-\hat{A}\hat{x}^{(k)}\right) 
\end{equation*}
where we recall that $\hat{A} = AM^{-1}$ and $\hat{x}=Mx$. Using our definition of $M = D_MV_M^T$ and the SVD of $A$, we may rewrite $\hat{A}$ as
\begin{equation*}
		\hat{A} = U_A\Sigma_AV^T_AV_MD_M^{-1}.
\end{equation*}
We may then write
\begin{equation*}
\hat{A}^T\hat{A} = D_M^{-1}V_M^TV_A\Sigma_A^T\Sigma_AV_A^TV_MD_M^{-1}
\end{equation*}
so that the $k^{th}$ expansion may be written as
\begin{equation} \label{prop1deriv}
	\begin{split}
		\hat{x}^{(k)}_{PL} &= \sum_{i=0}^{k-1} \left(I - \hat{A}^T\hat{A}\right)^i\hat{A}^Tb \\
		&= \sum_{i=0}^{k-1} \left(I - D_M^{-1}V_M^TV_A\Sigma_A^T\Sigma_AV_A^TV_MD_M^{-1}\right)^iD_M^{-1}V_M^TV_A\Sigma_A^TU_A^Tb.
	\end{split}
\end{equation}
Under our assumption, \eqref{prop1deriv} simplifies to
\begin{equation*}
	\hat{x}^{(k)}_{PL} =  \sum_{i=0}^{k-1} \left(I - D_M^{-1}\Sigma_A^T\Sigma_AD_M^{-1}\right)^iD_M^{-1}\Sigma_A^TU_A^Tb.
\end{equation*}
Denoting $\hat{D}^{k-1}=\sum_{i=0}^{k-1} \left(I - D_M^{-1}\Sigma_A^T\Sigma_AD_M^{-1}\right)^i$, we may write in terms of the $k^{th}$ iterative solution, $x^{(k)}$, by using that $x^{(k)}_{PL} = M^{-1}\hat{x}^{(k)}_{PL}$ so that
\begin{equation} \label{prop1_coroRelation}
	\begin{split}
		x^{(k)}_{PL} &= M^{-1}\hat{x}^{(k)}_{PL} \\
		&= V_MD_M^{-1}\hat{D}^{k-1}D_M^{-1}\Sigma_A^TU_A^Tb \\
		&= \sum_{j=1}^n \hat{d}_j^{k-1}\frac{\sigma_{A,j}}{d_{M,j}^2}\left(u_{A,j}^Tb\right)v_{M,j}
	\end{split}
\end{equation}
where $\hat{d}_j^{k-1} \in \mathbb{R}$ denotes the $j^{th}$ diagonal entry of $\hat{D}^{k-1}$ and may be written as a geometric series
\begin{equation*}
		\hat{d}_j^{k-1} = \sum_{i=0}^{k-1} \left(1 - \frac{\sigma_{A,j}^2}{d_{M,j}^2}\right)^i
		= \frac{1 - \left(1 - \frac{\sigma_{A,j}^2}{d_{M,j}^2}\right)^{k}}{\frac{\sigma_{A,j}^2}{d_{M,j}^2}}.
\end{equation*}
\rev{Additionally, $d_{M,j} = \left(\sigma^2_{M,j}+\alpha^2\right)^{1/2} \in \mathbb{R}$ denotes the $j^{th}$ entry from $D_M$.}
With this we may write $x^{(k)}_{PL}$ as
\begin{equation*}
	\begin{split}
		x^{(k)}_{PL} &= \sum_{j=1}^n \hat{d}_j^{k-1}\frac{\sigma_{A,j}}{d_{M,j}^2}\left(u_{A,j}^Tb\right)v_{M,j} \\
		&= \sum_{j=1}^n \left(1 - \left(1 - \frac{\sigma_{A,j}^2}{d_{M,j}^2}\right)^{k}\right)\frac{u_{A,j}^Tb}{\sigma_{A,j}}v_{M,j}.
	\end{split}
\end{equation*}
Taking the common denominator inside the term raised to the $k^{th}$ power and simplifying $d^2_{M,j}$ gives the filtered solution. \hfill $\qed$
\end{proof}

\noindent We mention several comments:
\begin{enumerate}

\item The solution of $x^{(k)}_{PL}$ is written using the basis vectors of $V_M$ instead of $V_A$ as in \eqref{filterLand}. Consider that we let $\hat{x}$ be the approximate solution of \eqref{preconLS} which can be written as a filtered solution \eqref{filteredSoln} using the SVD of $\hat{A}$. Then by definition of $M$, we can write $x=V_MD_M^{-1}\hat{x}$ showing that the basis choice of the solution method is independent of our orthogonality assumption.

\item The orthogonality assumption between the right singular vectors of $A$ and $M$ is used in all results herein and may be (approximately) reasonable if the precision of $M$ relative to $A$ is close enough when $M^TM$ represents a low precision approximation to $A^TA +\alpha^2I$. \rev{To this end, in our numerical experiments in Section \ref{sec5} we show that the filter factors behave in an approximate Tikhonov-type manner. Additionally, due to comment 1 above, we expect it to be less critical for the filter factors to be exact.}

\item For the remainder of this work we utilize the notation $\psi^{(k)}_j \in \mathbb{R}$ to denote the $j^{th}$ filter factor from the $k^{th}$ step of preconditioned Landweber and $\Psi^{(k)} \in \mathbb{R}^n$ to denote the vector containing the $n$ filter factors of the $k^{th}$ step.
\end{enumerate}
\section{Filtering Analysis of Iterative Refinement} \label{sec4}

In this section we derive the filter factors of IR on the Tikhonov problem using the filter factors from preconditioned Landweber. We first consider the simple case in Section \ref{sec4.1} where a Tikhonov-type $M$ may be constructed from an approximate SVD of $A$ that could model a low-precision preconditioner. In Section \ref{sec4.2} we extend our results to the determination of the filter factors in up to three precisions.

\subsection{Reformulated Iterative Refinement on the Tikhonov problem} \label{sec4.1}
We begin by stating the $k^{th}$ solution of IR on the Tikhonov problem with preconditioner $M$ written recursively:
\begin{equation} \label{IR_recursive}
    x^{(k)}_{IR} = x^{(k-1)}_{IR} + \left(M^TM\right)^{-1}A^Tr^{(k-1)} - \alpha^2 \left(M^TM\right)^{-1}x^{(k-1)}_{IR}.
\end{equation}
In Lemma \ref{IR_reform}, we show that the $k^{th}$ solution of IR on the Tikhonov problem may be written with respect to the $k^{th}$ solution of preconditioned Landweber given by \eqref{preconLand}.

\begin{lemma}{(Reformulation of IR on the Tikhonov problem)\\} \label{IR_reform}
    The $k^{th}$ solution of IR on the Tikhonov problem denoted by $x^{(k)}_{IR}$ may be written with respect to the $k^{th}$ solution of the preconditioned Landweber method \eqref{preconLand} denoted by $x^{(k)}_{PL}$. Precisely, we have that
    \begin{equation} \label{thm1_relation}
        \begin{split}
            x^{(k)}_{IR} &= x^{(k-1)}_{IR} + \left(M^TM\right)^{-1}A^Tr^{(k-1)} - \alpha^2 \left(M^TM\right)^{-1}x^{(k-1)}_{IR} \\
            &= x^{(k)}_{PL} - \alpha^2 \left(M^TM\right)^{-1}\sum_{i=0}^{k-1}\left(I-A^TA\left(M^TM\right)^{-1}\right)^ix^{(k-1-i)}_{IR}.
        \end{split}
    \end{equation}
\end{lemma}
\begin{proof}
We proceed via induction defining that $x^{(0)}_{IR} = x^{(0)}_{PL} = 0$. \\ \\
\emph{Base case $(k=1)$:} \\
Starting from the base relation of IR on the Tikhonov problem given by \eqref{IR_recursive} we have that
\begin{equation*}
    \begin{split}
        x^{(1)}_{IR} &= x^{(0)}_{IR} + \left(M^TM\right)^{-1}A^Tr^{(0)} - \alpha^2 \left(M^TM\right)^{-1}x^{(0)}_{IR} \\
        &= \underbrace{x^{(0)}_{PL} + \left(M^TM\right)^{-1}A^T\left(b-Ax^{(0)}_{PL}\right)}_{=x^{(1)}_{PL}}-\alpha^2 \left(M^TM\right)^{-1}x^{(0)}_{IR} \\
        &= x^{(1)}_{PL} - \alpha^2\left(M^TM\right)^{-1}\left(I-A^TA\left(M^TM\right)^{-1}\right)^0x^{(0)}_{IR} \\
        &= x^{(1)}_{PL} - \alpha^2\left(M^TM\right)^{-1}\sum_{i=0}^0 \left(I-A^TA\left(M^TM\right)^{-1}\right)^ix^{(0-i)}_{IR}
    \end{split}
\end{equation*}
completing the base case.\\ \\
\emph{Inductive step:} \\
Assume the inductive hypothesis is true for $k=n$ for $n>1$; that is
\begin{equation*}
    \begin{split}
        x^{(n)}_{IR} &= x^{(n-1)}_{IR} + \left(M^TM\right)^{-1}A^Tr^{(n-1)} - \alpha^2 \left(M^TM\right)^{-1}x^{(n-1)}_{IR} \\
        &= x^{(n)}_{PL} - \underbrace{\alpha^2 \left(M^TM\right)^{-1}\sum_{i=0}^{n-1}\left(I-A^TA\left(M^TM\right)^{-1}\right)^ix^{(n-1-i)}_{IR}}_{=\Gamma}.
    \end{split}
\end{equation*}
We proceed as follows
\begin{equation*}
    \begin{split}
        &x^{(n+1)}_{IR} = x^{(n)}_{IR} + \left(M^TM\right)^{-1}A^T\left(b-Ax^{(n)}_{IR}\right) - \alpha^2 \left(M^TM\right)^{-1}x^{(n)}_{IR} \\
        &= x^{(n)}_{IR} + \left(M^TM\right)^{-1}A^T\left(b-A\left[x^{(n)}_{PL} - \Gamma\right]\right) - \alpha^2 \left(M^TM\right)^{-1}x^{(n)}_{IR} \\
        &= x^{(n)}_{IR} + \underbrace{\left(M^TM\right)^{-1}A^T\left(b-Ax_{PL}^{(n)}\right)}_{=h^{(n)}} - \left(M^TM\right)^{-1}A^TA\Gamma - \alpha^2 \left(M^TM\right)^{-1}x^{(n)}_{IR} \\
        &= x^{(n)}_{PL} - \Gamma + h^{(n)} - \left(M^TM\right)^{-1}A^TA\Gamma - \alpha^2 \left(M^TM\right)^{-1}x^{(n)}_{IR} \\
        &= x^{(n+1)}_{PL} - \alpha^2\left(M^TM\right)^{-1}\left(I - A^TA\left(M^TM\right)^{-1}\right)^0x^{(n)}_{IR} \\
        & \hspace{5mm} - \alpha^2\left(M^TM\right)^{-1}\left(I - A^TA\left(M^TM\right)^{-1}\right)\sum_{i=0}^{n-1}\left(I - A^TA\left(M^TM\right)^{-1}\right)^ix^{(n-1-i)}_{IR} \\
        &= x^{(n+1)}_{PL} - \alpha^2 \left(M^TM\right)^{-1}\sum_{i=0}^{n}\left(I-A^TA\left(M^TM\right)^{-1}\right)^ix^{(n-i)}_{IR}.
    \end{split}
\end{equation*}
Since both the base case and the inductive step have been shown, by mathematical induction the relation \eqref{thm1_relation} holds for every natural number $k$.\hfill $\qed$
\end{proof}

With the result of Lemma \ref{IR_reform}, we may observe that the $k^{th}$ iterate of IR on the Tikhonov problem may be computed using the first $k$ solutions of \eqref{preconLand_wFF}. As such, it is possible to recursively determine the filter factors for $x^{(k)}_{IR}$. Consider the first iteration of IR using the result from Lemma \ref{IR_reform}:
\begin{equation*}
    \begin{split}
        x^{(1)}_{IR} &= x^{(1)}_{PL} - \alpha^2 (M^TM)^{-1}\sum_{i=0}^0 \left(I-A^TA(M^TM)^{-1}\right)^ix^{(0-i)}_{IR} \\
        &= x^{(1)}_{PL} = \sum_{j=1}^n \underbrace{\left(1 - \left(\frac{\sigma_{M,j}^2+\alpha^2-\sigma_{A,j}^2}{\sigma_{M,j}^2+\alpha^2}\right)^1\right)}_{=\phi_j^{(1)}}\frac{u_{A,j}^Tb}{\sigma_{A,j}}v_{M,j}.
    \end{split}
\end{equation*}
Here and throughout, we differentiate the filter factors of IR from preconditioned Landweber by defining that the filter factors of $x^{(k)}_{IR}$ are given by $\Phi_{IR}^{(k)} \in \mathbb{R}^n$ whose $j^{th}$ term we denoted by $\phi_j^{(k)}$. To determine the filter factors of $x^{(2)}_{IR}$ we proceed as follows using that $x^{(1)}_{IR} = x^{(1)}_{PL}$
\begin{equation} \label{coro_pt1}
    \begin{split}
        x^{(2)}_{IR} &= x^{(2)}_{PL} - \alpha^2 (M^TM)^{-1}\sum_{i=0}^1 \left(I-A^TA(M^TM)^{-1}\right)^ix^{(1-i)}_{IR} \\
        &= x^{(2)}_{PL} - \alpha^2 (M^TM)^{-1}\left(I-A^TA(M^TM)^{-1}\right)^0x^{(1)}_{PL}.
    \end{split}
\end{equation}

Before proceeding, we consider the general situation that will arise when expanding the $k^{th}$ solution of IR with respect to the solutions of preconditioned Landweber\rev{, i.e., we consider the simplification of
\begin{equation} \label{coro_pt2}
    \alpha^2 (M^TM)^{-1}\left(I-A^TA(M^TM)^{-1}\right)^ix^{(k)}_{PL}
\end{equation}
that arises from the ability to rewrite $x^{(k-1-i)}_{IR}$ from the summation term in \eqref{thm1_relation} in terms of components of preconditioned Landweber.}
We first note that 
$$\left(I-A^TA(M^TM)^{-1}\right)^i$$ 
may be decomposed using the definition of the preconditioner $M$, the SVD of $A$, and the orthgonality assumption between the right singular vectors of $A$ and of $M$:
\rev{\begin{equation*}
    \begin{split}
        \left(I-A^TA(M^TM)^{-1}\right)^i &= \left(I-V_A\Sigma_A^2V_A^TV_MD_M^{-2}V_M^T\right)^i \\
        &= V_A\left(I-\tilde{D}_M\right)^iV_M^T
    \end{split}
\end{equation*}
where $\tilde{D}_M = \Sigma_A^2D_M^{-2}= \text{diag}\left(\sigma_{A,1}^2/(\sigma_{M,1}^2+\alpha^2),\dots,\sigma_{A,n}^2/(\sigma_{M,n}^2+\alpha^2)\right)$.} With this, we may rewrite \eqref{coro_pt2} \rev{utilizing the first equality given in \eqref{prop1_coroRelation} so that}
    \rev{\begin{align*}
\alpha^2(M^TM)^{-1}&\left(I-A^TA(M^TM)^{-1}\right)^ix^{(k)}_{PL} \\
&= \alpha^2 V_M D_M^{-2} (I - \tilde{D}_M)^i D_M^{-2} \hat{D}^{k-1} D_M^{-2} \Sigma_A^T U_A^T b.
\end{align*}
In conjunction with the second equality of \eqref{prop1_coroRelation} we use that the $j^{th}$ entry of $\alpha^2D_M^{-2}(I-\tilde{D}_M)^i$ may be expressed as $\frac{\alpha^2}{\sigma^2_{M,j}+\alpha^2}\left(1-\frac{\sigma^2_{A,j}}{\sigma^2_{M,j}+\alpha^2}\right)^i$ to write \eqref{coro_pt2} as
\begin{equation*}
        \sum_{j=1}^n \frac{\alpha^2}{\sigma_{M,j}^2+\alpha^2}\left(1-\frac{\sigma_{A,j}^2}{\sigma_{M,j}^2+\alpha^2}\right)^i\left(1 - \left(\frac{\sigma_{M,j}^2+\alpha^2-\sigma_{A,j}^2}{\sigma_{M,j}^2+\alpha^2}\right)^k\right)\frac{u_{A,j}^Tb}{\sigma_{A,j}}v_{M,j}.
    \end{equation*}}

Returning to \eqref{coro_pt1}, we may rewrite $x^{(2)}_{IR}$ as follows
\begin{equation*}
    \begin{split}
        x^{(2)}_{IR} &= x^{(2)}_{PL} - \alpha^2 (M^TM)^{-1}\left(I-A^TA(M^TM)^{-1}\right)^0x^{(1)}_{PL} \\
        &= \sum_{j=1}^n \phi^{(2)}_{j} \frac{u_{A,j}^Tb}{\sigma_{A,j}}v_{M,j}
    \end{split}
\end{equation*}
where
\begin{equation*}
    \begin{split}
        \phi^{(2)}_{j} &= \psi^{(2)}_j -
        \frac{\alpha^2}{\sigma_{M,j}^2+\alpha^2}\left(1-\frac{\sigma_{A,j}^2}{\sigma_{M,j}^2+\alpha^2}\right)^0\left(1 - \left(\frac{\sigma_{M,j}^2+\alpha^2-\sigma_{A,j}^2}{\sigma_{M,j}^2+\alpha^2}\right)^1\right)\\
        &= \rev{\psi^{(2)}_j - \frac{\alpha^2}{\sigma_{M,j}^2+\alpha^2}\left(1 - \left(\frac{\sigma_{M,j}^2+\alpha^2-\sigma_{A,j}^2}{\sigma_{M,j}^2+\alpha^2}\right)^1\right)}\\
    \end{split}
\end{equation*}
for $j=1,2,\dots,n$. Here, we note that $\Phi^{(2)}$ is the difference between the filter factors of $x^{(2)}_{PL}$ and a scaling of $\Phi^{(1)}$. Continuing this recursive process, we may arrive at the following general result.

\begin{theorem}{(IR on the Tikhonov problem with filter factors)} \label{thm1} \\
    The $k^{th}$ solution of IR on the Tikhonov problem, $x^{(k)}_{IR}$, may be written as \Rev{
    \begin{equation*}
        x^{(k)}_{IR} = \sum_{j=1}^n \phi_{j}^{(k)}\frac{u_{A,j}^Tb}{\sigma_{A,j}}v_{M,j}
    \end{equation*}
    where the $j^{th}$ filter factor per $k^{th}$ step, $\phi_{j}^{(k)}$,
    is given by} 
    \begin{equation} \label{IR_DP_FF}
    \phi_{j}^{(k)} = \psi_j^{(k)} - \left(\frac{\alpha^2}{\sigma_{M,j}^2+\alpha^2}\right)\sum_{i=0}^{k-1}\left(1 - \frac{\sigma_{A,j}^2}{\sigma_{M,j}^2+\alpha^2}\right)^i\phi_{j}^{(k-1-i)}, \quad j=1,2,\dots,n
    \end{equation}
    under the assumption that $V_M^TV_A = I = V_A^TV_M$.
\end{theorem}

We comment that the filter factors of IR on the Tikhonov problem computed in exact arithmetic when $M^TM=A^TA+\alpha^2I$ are exactly those of Tikhonov regularization. This is stated in Corollary \ref{TikhonovFF_IR} below without proof as the argument follows from a straight forward mathematical induction argument.

\begin{corollary}{(Filter factors of IR on the Tikhonov problem in exact arithmetic)} \label{TikhonovFF_IR} \\
    In exact arithmetic the filter factors of IR on the Tikhonov problem coincide with the filter factors of standard Tikhonov regularization \eqref{TikClosedForm} for all $k=0,1,2,\dots$ given by
    \begin{equation} \label{TikFF}
    \phi^{(k)}_j = \frac{\sigma^2_j}{\sigma^2_j+\alpha^2}, \quad \text{for}\,\,\, j=1,2,\dots,n.
    \end{equation}
\end{corollary}

\subsection{Mixed precision IR on the Tikhonov problem} \label{sec4.2}

We now turn to the determination of the filter factors of IR when computing in multiple precisions. We extend the result of Theorem \ref{thm1} to the case when the $k^{th}$ iterate and its corresponding filter factors are to be computed in up to three precisions. When discussing computations in mixed precision we parallel the convention of recent works, e.g., \cite{Higham1,Higham22}, by using the notation 
\begin{equation*}
    \text{Pr}_1 \leq \text{Pr}_2 \leq \text{Pr}_3
\end{equation*}
to represent the three precisions used. Here, Pr$_3$ denotes the highest precision that is used in residual updates, Pr$_2$ denotes the `working precision' used for linear systems solves, and Pr$_1$ denotes the precision used to represent the preconditioner $M$. In Section \ref{sec5} we define the precisions we consider for our numerical experiments. \Rev{We comment that a finite precision error analysis could be considered to better understand the error accumulation from use of three different precisions; such a treatment is beyond the scope of this work.}

\rev{We now consider the reformulation of IR on the Tikhonov problem in mixed precision so as to determine its corresponding filter factors. In three precisions the $k^{th}$ solution of \eqref{IR_recursive} may formulated as
\begin{equation} \label{MP_IR_relation}
    x^{(k)}_{IR} = \underbrace{x^{(k-1)}_{IR} + {\underbrace{\vphantom{\Big|}\left(M^TM\right)}_{\text{Pr}_1}}^{-1}\underbrace{\left[A^Tr^{(k-1)} - \alpha^2 x^{(k-1)}_{IR}\right]}_{\text{Pr}_3}}_{\text{Pr}_2}
\end{equation}
where we utilize braces to visually denote the usage of mixed precision.
This mixed precision variant of Algorithm \ref{alg_IRTik} is provided in Algorithm \ref{alg_3lvlIR} for reference. Recall that the $k^{th}$ solution of \eqref{IR_recursive} can be written as a recursive relation involving the $k^{th}$ iterate of preconditoned Landweber and the previous $k-1$ iterates of IR on the Tikhonov problem. In Lemma \ref{IR_reform_MP} we reformulate \eqref{IR_recursive} to write the $k^{th}$ iterate of the IR problem as a filtered solution whose constituent parts are computed in precisions that mimic their partnered forms of \eqref{MP_IR_relation} and Algorithm \ref{alg_3lvlIR}.}

\begin{algorithm}
\caption{Mixed precision iterative refinement on the Tikhonov problem}
\label{alg_3lvlIR}
\algsetup{indent=2em,linenodelimiter=}
\begin{algorithmic}[1]
\STATE{\text{\bf{Precision Levels:}} $\text{Pr}_1 \leq \text{Pr}_2 \leq \text{Pr}_3$}
\STATE{\rev{\text{\bf{Input:}} $A \in \mathbb{R}^{m \times n},\,\, x_{IR}^{(0)} \in \mathbb{R}^n$}}
\STATE{\rev{\text{\bf{Output:}} $x_{IR}^{(k+1)} \in \mathbb{R}^n$}}
\STATE{Compute $M^TM = \left(A^TA + \alpha^2 I\right);$ \hfill $[\text{in } \text{Pr}_1]$}
\FOR{$k = 0,1,2,\dots$}
    \STATE{\rev{$r^{(k)} = b - Ax_{IR}^{(k)};$ \hfill $[\text{in } \text{Pr}_3]$}}
    \STATE{\rev{$s^{(k)} = A^Tr^{(k)}-\alpha^2x_{IR}^{(k)};$ \hfill $[\text{in } \text{Pr}_3]$}}
    \STATE{Solve $\left(M^TM\right)h^{(k)} = s^{(k)};$ \hfill $[\text{in } \text{Pr}_2]$}
    \STATE{\rev{$x_{IR}^{(k+1)} = x_{IR}^{(k)} + h^{(k)};$ \hfill $[\text{in } \text{Pr}_2]\,\,$}}
\ENDFOR
\end{algorithmic}
\end{algorithm}

\begin{lemma}{(Reformulation of IR on the Tikhonov problem for mixed precision)\\} \label{IR_reform_MP}
    The $k^{th}$ solution of IR on the Tikhonov problem, $x_{IR}^{(k)}$, in precisions Pr$_1$, Pr$_2$, and Pr$_3$ may be written with respect to the $k^{th}$ solution of \rev{preconditioned Landweber}, $x^{(k)}_{PL}$. \rev{Precisely,
    \begin{equation} \label{lemma2_result}
        \begin{split}
            &x^{(k)}_{IR} = x^{(k-1)}_{IR} + \left(M^TM\right)^{-1}A^Tr^{(k-1)} - \alpha^2 \left(M^TM\right)^{-1}x^{(k-1)}_{IR} \\
            &= \underbrace{x^{(k-1)}_{IR} + {\underbrace{\vphantom{\Big|}\left(M^TM\right)}_{\text{Pr}_1}}^{-1}\underbrace{\left[\left(M^TM\right)\left[x^{(k)}_{PL} - x^{(k-1)}_{PL}\right]+ \Gamma-\alpha^2 x_{IR}^{(k-1)}\right]}_{\text{Pr}_3}}_{\text{Pr}_2}
        \end{split}
    \end{equation}
    where} $\Gamma = A^TA\alpha^2 \left(M^TM\right)^{-1}\sum_{i=0}^{k-2}\left(I - A^TA\left(M^TM\right)^{-1}\right)^ix_{IR}^{(k-2-i)}$.
\end{lemma}

\begin{proof}
We begin by considering the IR expansion about its $(k-1)^{th}$ residual
\begin{equation} \label{thm2_eq1}
        x^{(k)}_{IR} = x^{(k-1)}_{IR} + \left(M^TM\right)^{-1}A^T\left(b-Ax_{IR}^{(k-1)}\right) - \alpha^2 \left(M^TM\right)^{-1}x^{(k-1)}_{IR}.
\end{equation}
Using Lemma \ref{IR_reform} we may rewrite the residual of $x^{(k-1)}_{IR}$ in \eqref{thm2_eq1} instead with respect to $x^{(k-1)}_{PL}$
\begin{equation*} \label{thm2_eq2}
    \begin{split}
        x^{(k)}_{IR} &= x^{(k-1)}_{IR} + \left(M^TM\right)^{-1}A^T\left(b-Ax_{IR}^{(k-1)}\right) - \alpha^2 \left(M^TM\right)^{-1}x^{(k-1)}_{IR} \\
        &= x^{(k-1)}_{IR} + \left(M^TM\right)^{-1}A^Tr_{PL}^{(k-1)} + \left(M^TM\right)^{-1}\Gamma - \alpha^2 \left(M^TM\right)^{-1}x^{(k-1)}_{IR}\\
        &= x^{(k-1)}_{IR} + \left(M^TM\right)^{-1}\left[A^Tr_{PL}^{(k-1)} +\Gamma -\alpha^2 x^{(k-1)}_{IR} \right]
    \end{split}
\end{equation*}
where $r^{(k-1)}_{PL}$ denotes the $(k-1)^{th}$ residual of preconditioned Landweber solution given by \eqref{preconLand_wFF_3prec}. To compute $x^{(k)}_{IR}$ with respect to its own prior solutions and $x_{PL}^{(q)}$ for $q=0,1,\dots,k$, we rewrite $A^Tr_{PL}^{(k-1)}$ using \eqref{preconLand} as follows
\begin{equation*}
    x_{PL}^{(k)} = x_{PL}^{k-1} + \left(M^TM\right)^{-1}A^Tr_{PL}^{(k-1)} \\
    \iff \\
    A^Tr_{PL}^{(k-1)} = \left(M^TM\right)\left[x_{PL}^{(k)}-x_{PL}^{(k-1)}\right].
\end{equation*}
Here, we point out that by construction $M^TM$ is invertible. With this substitution the result follows. \hfill $\qed$
\end{proof}

We comment that during the computation of the $k^{th}$ iterate of \eqref{lemma2_result} that the product $\left(M^TM\right)\left[x^{(k)}_{PL}-x^{(k-1)}_{PL}\right]$ is computed in Pr$_3$ but $\left(M^TM\right)$ is computed in Pr$_1$. Thus, before presenting the mixed precision generalization of Theorem \ref{thm1}, we present in Proposition \ref{prop2} the \Rev{mixed precision version of} Proposition \ref{prop1}. In Algorithm \ref{alg_mixedPrecPreconLand} we provide the steps for computing preconditioned Landweber where we note that the preconditioner $\left(M^TM\right)$ is an input and \Rev{Pr$_3$ is used uniformly (as per Lemma \ref{IR_reform_MP})}. The framework for the computations of the filter factors is summarized in Algorithm \ref{alg_5_preconLand_3prec}.

\begin{algorithm}
\caption{\Rev{Mixed} precision preconditioned Landweber}
\label{alg_mixedPrecPreconLand}
\algsetup{indent=2em,linenodelimiter=}
\begin{algorithmic}[1]
\STATE{\text{\bf{Precision Levels:}} $\text{Pr}_1 \leq \text{Pr}_2 \leq \text{Pr}_3$}
\STATE{\rev{\text{\bf{Input:}} $A \in \mathbb{R}^{m \times n},\,\, b\in \mathbb{R}^m,\,\, x_{PL}^{(0)} = 0 \in \mathbb{R}^n,\,\, \alpha>0$,\,\, $M^TM$ \text{ (in Pr$_1$)}}}\\
\STATE{\rev{\text{\bf{Output:}} $x_{PL}^{(k+1)} \in \mathbb{R}^n$}}
\FOR{$k = 0,1,2,\dots$}
    \STATE{\rev{$s^{(k)} = A^T\left(b - Ax_{PL}^{(k)}\right)$\hfill $[\text{in } \text{Pr}_3]$}}
    \STATE{\rev{Solve $(M^TM)h^{(k)} = s^{(k)}$ \hfill $[\text{in } \text{Pr}_3]$}}
    \STATE{\rev{$x_{PL}^{(k+1)} = x_{PL}^{(k)} + h^{(k)}$ \hfill $[\text{in } \text{Pr}_3]\,\,$}}
\ENDFOR
\end{algorithmic}
\end{algorithm}

\begin{prop}{(\Rev{Mixed precision} preconditioned Landweber with filter factors)\\} \label{prop2}
The $k^{th}$ solution, $x^{(k)}_{PL}$, of \eqref{preconLand} \rev{computed in Pr$_3$} with preconditioner $M = D_MV_M^T$ as defined in Section \ref{sec2} \rev{such that $M^TM$ is computed in Pr$_1$} may be written as 
\begin{equation} \label{preconLand_wFF_3prec}
	\Rev{x^{(k)}_{PL}
        = \sum_{j=1}^n \psi_j^{(k)}\frac{u^T_{A,j}b}{\sigma_{A,j}}v_{M,j}}
\end{equation}
with filter factors per $k^{th}$ step given by 
\begin{equation} \label{FF_preconLand3prec}
\Rev{
\psi_j^{(k)} = \underbrace{\psi_j^{(k-1)} + \sigma_{A,j}^2
(\underbrace{\sigma^2_{M,j}+\alpha^2}_{Pr_1})^{-1}
\Big(1-\sigma_{A,j}^2(\sigma^2_{M,j}+\alpha^2)^{-1}\Big)^{k-1}}_{Pr_3}
}
\end{equation}
\Rev{for $j=1,2,\dots,n$} under the assumption that $V_M^TV_A = I = V_A^TV_M$.
\end{prop}

\begin{proof}
From \eqref{preconLand} we have that 
\begin{equation*}
    x^{(k)}_{PL} = x^{(k-1)}_{PL} + \left(M^TM\right)^{-1}A^Tr^{(k-1)}.
\end{equation*}
To write the $k^{th}$ iterate as a filtered solution we rewrite $A^Tr^{(k-1)}$ while retaining that it be computed in Pr$_3$ from Algorithm \ref{alg_mixedPrecPreconLand}. Using \eqref{preconLand} and \eqref{prop1deriv} with $\zeta=1$ we may write
\begin{equation*}
    \begin{split}
        A^Tr^{(k-1)} &= \left(M^TM\right)\left(x^{(k)}_{PL} - x^{(k-1)}_{PL}\right)\\
        &= \left(M^TM\right)\left[M^{-1}\left(\hat{x}_{PL}^{(k)}-\hat{x}_{PL}^{(k-1)}\right)\right] \\
        &= \left(M^TM\right)\left[M^{-1}\left(\sum_{i=0}^{k-1} \left(I - \hat{A}^T\hat{A}\right)^i\hat{A}^Tb - \sum_{i=0}^{k-2} \left(I - \hat{A}^T\hat{A}\right)^i\hat{A}^Tb\right)\right]\\
        &= \left(M^TM\right)M^{-1}\left(I - \hat{A}^T\hat{A}\right)^{k-1}\hat{A}^Tb.
    \end{split}
\end{equation*}
With the definition of $M$, $\hat{A}$, and the SVD of $A$ we can rewrite preconditioned Landweber using our orthogonality assumption as follows
\begin{equation*}
    \begin{split}
        x^{(k)}_{PL} &= x^{(k-1)}_{PL} + \left(M^TM\right)^{-1}A^Tr^{(k-1)}\\
        &= x^{(k-1)}_{PL} + V_MD_M^{-2}D_M\left(I-D_M^{-1}\Sigma^2_AD_M^{-1}\right)^{k-1}D_M^{-1}\Sigma_A^TU_A^Tb.
    \end{split}
\end{equation*}
By defining the $j^{th}$ component of $\left(I-D_M^{-1}\Sigma^2_AD_M^{-1}\right)^{k-1}$ by $\tilde{d}_{j}^{k-1}=(1-\sigma_{A,j}^2/\sigma^2_{M,j}+\alpha^2)^{k-1}$ which come from the computation of $A^Tr^{(k-1)}$, we may write
\begin{equation*} \label{something}
    x^{(k)}_{PL} = x^{(k-1)}_{PL} + \sum_{j=1}^n \left(\left(d^2_{M,j}\right)^{-1}\sigma^2_{A,j}\tilde{d}^{k-1}_j\right)\frac{u^T_{A,j}b}{\sigma_{A,j}}v_{M,j}.
\end{equation*}
Here, $d_{M,j}^2=(\sigma^2_{M,j}+\alpha^2)$ represents the components from the preconditioner computed in Pr$_1$ and whose component-wise action \rev{is computed in Pr$_3$.} Simplifying, we have
\begin{equation*}
        x^{(k)}_{PL}
        = \sum_{j=1}^n \left(\psi_j^{(k-1)} + \frac{\sigma_{A,j}^2}{\sigma^2_{M,j}+\alpha^2}\left(1-\frac{\sigma_{A,j}^2}{\sigma^2_{M,j}+\alpha^2}\right)^{k-1}\right)\frac{u^T_{A,j}b}{\sigma_{A,j}}v_{M,j}
\end{equation*}
where the update to $\psi_j^{(k-1)}$ is \rev{also computed in Pr$_3$.} \hfill $\qed$
\end{proof}

\begin{algorithm}
\caption{Filter factors of \Rev{mixed precision} preconditioned Landweber}
\label{alg_5_preconLand_3prec}
\algsetup{indent=2em,linenodelimiter=}
\begin{algorithmic}[1]
\STATE{\text{\bf{Precision Levels:}} $\text{Pr}_1 \leq \text{Pr}_2 \leq \text{Pr}_3$}
\STATE{\text{\bf{Input:}} \rev{$\Psi^{(0)}=0 \in \mathbb{R}^{n},\,\,\Sigma^2_A\in \mathbb{R}^{n\times n},\,\, \alpha>0$,}}
\STATE{\phantom{\hspace{1.3cm}}\rev{$\left(\sigma_{M,j}^2+\alpha^2\right)\in \mathbb{R},\quad \text{for }j=1,2,\dots,n$ \text{ (in Pr$_1$)}}}
\STATE{\text{\bf{Output:}} $\Psi^{(k+1)} \in \mathbb{R}^n$}
\FOR{$k = 0,1,2,\dots$}
\FOR{$j = 1,2,\dots,n$}
\STATE{\rev{$q_j^{(k)} = \sigma_{A,j}^2\left(1-\sigma^2_{A,j}\left(\sigma_{M,j}^2+\alpha^2\right)^{-1}\right)^{k}$ \hfill $[\text{in } \text{Pr}_3]$}}
\STATE{\rev{$g_j^{(k)} = \left(\sigma_{M,j}^2+\alpha^2\right)^{-1}q_j^{(k)}$ \hfill $[\text{in } \text{Pr}_3]$}}
\STATE{\rev{$\psi_j^{(k+1)} = \psi_j^{(k)} + g_j^{(k)}$ \hfill $[\text{in } \text{Pr}_3]\,\,$}}
\ENDFOR
\ENDFOR
\end{algorithmic}
\end{algorithm}

\rev{We comment that the filter factors $\psi^{(k)}_j$ for $j=1,2,\dots,n$ computed by Proposition \ref{prop2} are equivalent to those of Proposition \ref{prop1} in exact arithmetic. This can be observed by using relations \eqref{FF_preconLand} and \eqref{FF_preconLand3prec} to consider the difference between subsequent filter factor iterates, i.e., $\psi^{(k)}_j - \psi^{(k-1)}_j$. It's worthwhile to mention that that the recursive nature of the filter factors from Proposition \ref{prop2} is a product of the structure that arises from Algorithm \ref{alg_mixedPrecPreconLand}.}

We now describe the computation of the filter factors of IR on the Tikhonov problem in precisions Pr$_1$, Pr$_2$, and Pr$_3$. The result is provided without proof as the derivations follow in a similar manner to Theorem \ref{thm1} when starting with the result of Lemma \ref{IR_reform}. Here, we adhere to the same \rev{appropriate precision choices described at the start of the section in \eqref{MP_IR_relation} with respect to the computation of residuals, linear systems solves, and the preconditioner.}  Algorithm \ref{alg_4FFmultPrec} summarizes the framework for computing the filter factors.

\begin{theorem}{(Mixed precision IR on the Tikhonov problem with filter factors)\\} \label{thm2}
    The $k^{th}$ solution of IR on the Tikhonov problem, $x^{(k)}_{IR}$, in precisions Pr$_1$, Pr$_2$, and Pr$_3$ may be written as
    \begin{equation*}
        x^{(k)}_{IR} = \sum_{j=1}^n \phi_j^{(k)}\frac{u^T_{A,j}b}{\sigma_{A,j}}v_{M,j}
    \end{equation*}
    with the $j^{th}$ filter factor per $k^{th}$ iterative step and indicated precision given by
    \begin{align} \label{IR_MP_FF}
    \rev{\phi_{j}^{(k)} = \underbrace{\phi_{j}^{(k-1)} + {\underbrace{\vphantom{\Big|}\left(\sigma_{M,j}^2+\alpha^2\right)}_{\text{Pr}_1}}^{-1}\underbrace{\Biggl[\left(\sigma_{M,j}^2+\alpha^2\right)\left[\psi_j^{(k)}-\psi_j^{(k-1)}\right] + \Delta -\alpha^2\phi^{(k-1)}_{j}\Biggr]}_{\text{Pr}_3}}_{\text{Pr}_2}}
    \end{align}
    \rev{where 
    $$\Delta = \left(\frac{\alpha^2\sigma_{A,j}^2}{\sigma_{M,j}^2+\alpha^2}\right)\sum_{i=0}^{k-2}\left(1-\frac{\sigma_{A,j}^2}{\sigma_{M,j}^2+\alpha^2}\right)^i\phi_{j}^{(k-2-i)}$$}
    for $j=1,2,\dots,n$ under the assumption that $V_M^TV_A = I = V_A^TV_M$.
\end{theorem}

\begin{algorithm}
\caption{Mixed precision filter factors of iterative refinement on the Tikhonov problem}
\label{alg_4FFmultPrec}
\algsetup{indent=2em,linenodelimiter=}
\begin{algorithmic}[1]
\STATE{\text{\bf{Precision Levels:}} $\text{Pr}_1 \leq \text{Pr}_2 \leq \text{Pr}_3$}
\STATE{\text{\bf{Input:}} $\Phi^{(0)}=0 \in \mathbb{R}^{n},\,\, \Psi^{(0)}=0 \in \mathbb{R}^{n},\,\,\Sigma^2_A\in \mathbb{R}^{n \times n},\,\,\alpha>0$}
\STATE{\text{\bf{Output:}} $\Phi^{(k+1)} \in \mathbb{R}^n$}
\STATE{\rev{Compute $\left(\sigma_{M,j}^2+\alpha^2\right)\in \mathbb{R},\,j=1,\dots,n$ from $M^TM=\left(A^TA+\alpha^2I\right)$ \hfill $[\text{in } \text{Pr}_1]$}}
\FOR{$k = 0,1,2,\dots$}
\STATE{\rev{Compute $\Psi^{(k+1)}$ using \textbf{Algorithm \ref{alg_5_preconLand_3prec}} \hfill $[\text{in } \text{Pr}_3]$}}
\FOR{$j = 1,2,\dots,n$}
    \STATE{$q_j^{(k)} = \left(\sigma_{M,j}^2+\alpha^2\right)\left[\psi_j^{(k+1)}-\psi_j^{(k)}\right] -\alpha^2\phi^{(k-1)}_{j}$ \hfill $[\text{in } \text{Pr}_3]$}
    \STATE{$q_j^{(k)} = q_j^{(k)} +\left(\frac{\alpha^2\sigma_{A,j}^2}{\sigma_{M,j}^2+\alpha^2}\right)\sum_{i=0}^{k-2}\left(1-\frac{\sigma_{A,j}^2}{\sigma_{M,j}^2+\alpha^2}\right)^i\phi_{j}^{(k-2-i)}$ \hfill $[\text{in } \text{Pr}_3]$}
    \STATE{$g_j^{(k)} = \left(\sigma_{M,j}^2+\alpha^2\right)^{-1}q_j^{(k)}$ \hfill $[\text{in } \text{Pr}_2]$}
    \STATE{$\phi_{j}^{(k+1)} = \phi_{j}^{(k)} + g_j^{(k)}$ \hfill $[\text{in } \text{Pr}_2]$\,\,}
\ENDFOR
\ENDFOR
\end{algorithmic}
\end{algorithm}

\section{Numerical Results and Preliminaries} \label{sec5}

In this section we provide numerical experiments to illustrate the effectiveness of mixed precision IR (MP-IR) on the Tikhonov problem and provide an exposition of the filter factor results discussed in Section \ref{sec4}. The section is organized as follows: Sections \ref{secNum2} and \ref{secNum3} provide the background for the determination of what we will refer to as effective filter factors of IR on the Tikhonov problem and an overview of the simulation of low precision computations in MATLAB, respectively. In Section \ref{secNum4} we discuss our numerical results and their associated preliminaries.

\subsection{Effective filter factors} \label{secNum2}

To assess the experimental validity of our filter factor results from Section \ref{sec4} for IR on the Tikhonov problem we compare against experimentally derived filter factors which we term \emph{effective filter factors}. Using \eqref{filteredSoln}, the effective filter factors $\omega^{(k)}_j$ of an approximate solution at the $k^{th}$ iteration may be computed as follows
\begin{equation} \label{effectiveFF}
    \omega^{(k)}_j = \frac{v_{M,j}^Tx^{(k)}_{IR}}{u^T_{A,j}b}\sigma_{A,j}, \quad j=1,2,\dots,n
\end{equation}
where we denote the vector containing the filter factors of the $k^{th}$ iterate by $\Omega^{(k)} \in \mathbb{R}^n$. We highlight that \eqref{effectiveFF} uses the right singular vectors of the preconditioner $M$ as a basis of the $k^{th}$ iterate which we commented on in our remarks following Proposition \ref{prop1}. In Section \ref{sec4} we showed that this also held true in Theorems \ref{thm1} and \ref{thm2}.

\subsection{Low precision simulation in MATLAB} \label{secNum3}

To ascertain the effectiveness of mixed precision IR on the Tikhonov problem and the experimental validity of the results presented in Section \ref{sec4} we utilize the software package \textbf{chop} introduced by Higham and Pranesh in \cite{HighamPranesh19}. The \textbf{chop} function simulates lower precision arithmetic by rounding array entries given in a MATLAB native precision (e.g., single or double) to a target precision which is stored in a higher precision with non-utilized representation bits set to zero. Chen et al. in \cite{Nagy23} considered the application of \textbf{chop} for the solution of structured inverse problems and investigated associated numerical considerations of its usage.

Various target precisions are supported by the \textbf{chop} software including user-customizable ones. The precisions we utilize in our numerical experiments in Section \ref{secNum4} and their associated shorthand notations are provided in Table \ref{choptable}. We comment that we did not experience any significant experimental differences when using precisions that did not support subnormal numbers, e.g., Google's bfloat16\footnote{https://en.wikipedia.org/wiki/Bfloat16$\_$floating-point$\_$format}, compared to those due to the IEEE standard\footnote{https://en.wikipedia.org/wiki/IEEE$\_$754$\#$2019} during our numerical investigations and therefore do not include them in our results.

\begin{table}[ht]
	\centering
    \caption{\rev{Precisions considered in the numerical results: shorthand given as an integer, bits given for the mantissa and exponent, the approximate decimal round-off error, and whether precision is IEEE standard or not.}}
    \begin{tabular}{ c | c | c | c | c | c }
        \makecell{Pr\\Shorthand} & Name & \makecell{Exponent\\Bits} & \makecell{Mantissa\\Bits} & \makecell{\rev{Approx. Decimal}\\\rev{Round-off}} & \makecell{IEEE\\Standard} \\
        \ChangeRT{1.5pt}
        $1$ & fp64 & $11$ & $52$ & $1.11e\text{-}16$ & Yes \\
        $2$ & fp32 & $8$ & $23$ & $5.96e\text{-}8$ & Yes \\
        $3$ & fp16 & $5$ & $10$ & $4.88e\text{-}4$ & Yes \\
    \end{tabular}
    \label{choptable}
\end{table}

The reported numerical results of this work were carried out in MATLAB R2022b 64-bit on a MacBook Pro laptop 
running MacOS Ventura with an Apple M2 Pro processor with @3.49 GHz and 16 GB of RAM. The computations other than those utilizing \textbf{chop} were carried out with about 15 significant decimal digits.

\subsection{Examples and Results} \label{secNum4}

For our first example we consider the 1D signal restoration problem \emph{Spectra} whose matrix models a symmetric Gaussian blur and $x$ is a simulated X-ray spectrum \cite{trussell}. The $a_{i,j}$ entries of $A\in \mathbb{R}^{64 \times 64}$ are given by
\begin{equation*} \label{spectraEntries}
    a_{i,j} = \frac{1}{\eta\sqrt{2\pi}}\exp\left(-\frac{(i-j)^2}{2\eta^2}\right),
\end{equation*}
with $\eta=2$ which results in a Toeplitz matrix.

To realistically simulate the inverse problem, noise was added to the true right-hand side, $\hat{b}$, by forming the vector $e$ with normally distributed random entries with mean zero so that $b=Ax + e$; the vector $e$ is scaled so as to correspond to a specific noise level given by $\mu = 100\left(\|e\|/\|\hat{b}\|\right)$. We will refer to $\mu$ as the noise level. The condition number of the matrix $A$ as determined by the MATLAB function \mat{cond()} is $\approx10^9$; it can also be easily verified that the singular values of this matrix decay without a significant gap.

\rev{To efficiently use \textbf{chop} in an experimental setting to simulate low precision arithmetic, the matricies $A$ we utilize herein for the numerical examples are Toeplitz structured, which, for inverse problems are often well approximated by circulant matrices; see e.g., \cite{Estatico08, DH13}. Because of this,} we compare the solution quality of MP-IR to what we will term the \emph{Approximated Iterative Refinement} (AIR) \rev{method where we use circulant approximations of $A$ within the preconditioner $M^TM$ to construct a normal preconditioning matrix which admits a spectral decomposition that can be computed efficiently. Specifically,} we let the preconditioner be $M^TM = \left(C^TC + \alpha^2I\right)$ where $C$ represents a circulant approximation to $A$. We approximate $C\approx A$ using the scheme devised by Chan in \cite{ChanOptimal} which minimizes the Frobenius norm of the difference between $A$ and $C$ amongst a family of circulants. This approximation scheme was analyzed by Strela and Tyrtyshnikov \cite{ST96} and found to be a good choice from the viewpoint of eigenvalue clustering when the vector to be reconstructed takes on many zero values.

(\textbf{\emph{Spectra - filter factors in fp64}}) - We begin \rev{with a comparison of our filter factor results presented herein} by evaluating the filtering behavior of MP-IR on the \emph{Spectra} problem contaminated by $1\%$ noise in fp64. For this experiment we chose the regularization parameter $\alpha^2 = 1e\text{-}2$; in the restoration part of this example we expand on this choice. Figure \ref{resultsFF_DP} displays the filter factors from Theorem \ref{thm1}, Theorem \ref{thm2}, and the effective filter factors compared to those of Tikhonov \eqref{TikFF} when considering all computations in fp64 after $1$ iteration. We comment that there is no visual discernible difference amongst the plots thereby experimentally verifying that: (i) the filter factors of Theorem \ref{thm1} and Theorem \ref{thm2} coincide when all floating point computations are done in fp64, i.e., when $(\text{Pr}_1,\,\text{Pr}_2,\,\text{Pr}_3) = (1,1,1)$ and (ii) that all three sets of filter factors match those of Tikhonov. Table \ref{table_FFs} numerically confirms the agreement \rev{to unit round-off} between Theorem \ref{thm2} and the effective filter factors in fp64 working precision.

\begin{figure}[ht]
\centering
\begin{minipage}{1\linewidth}
	\centering
	\begin{minipage}{0.32\linewidth}
		\centering
		\includegraphics[trim={6cm 0 6cm 0},clip,width=\linewidth]{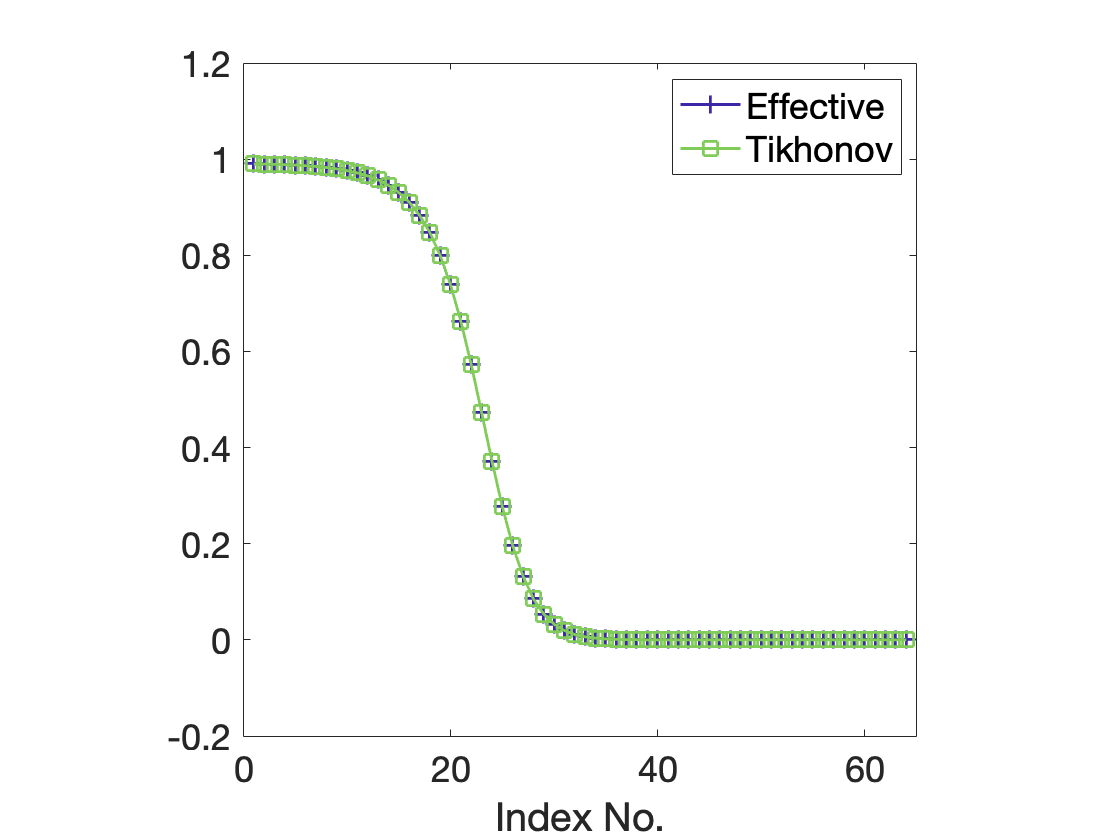}\\(a)
	\end{minipage}
	\begin{minipage}{0.32\linewidth}
		\centering
		\includegraphics[trim={6cm 0 6cm 0},clip,width=\linewidth]{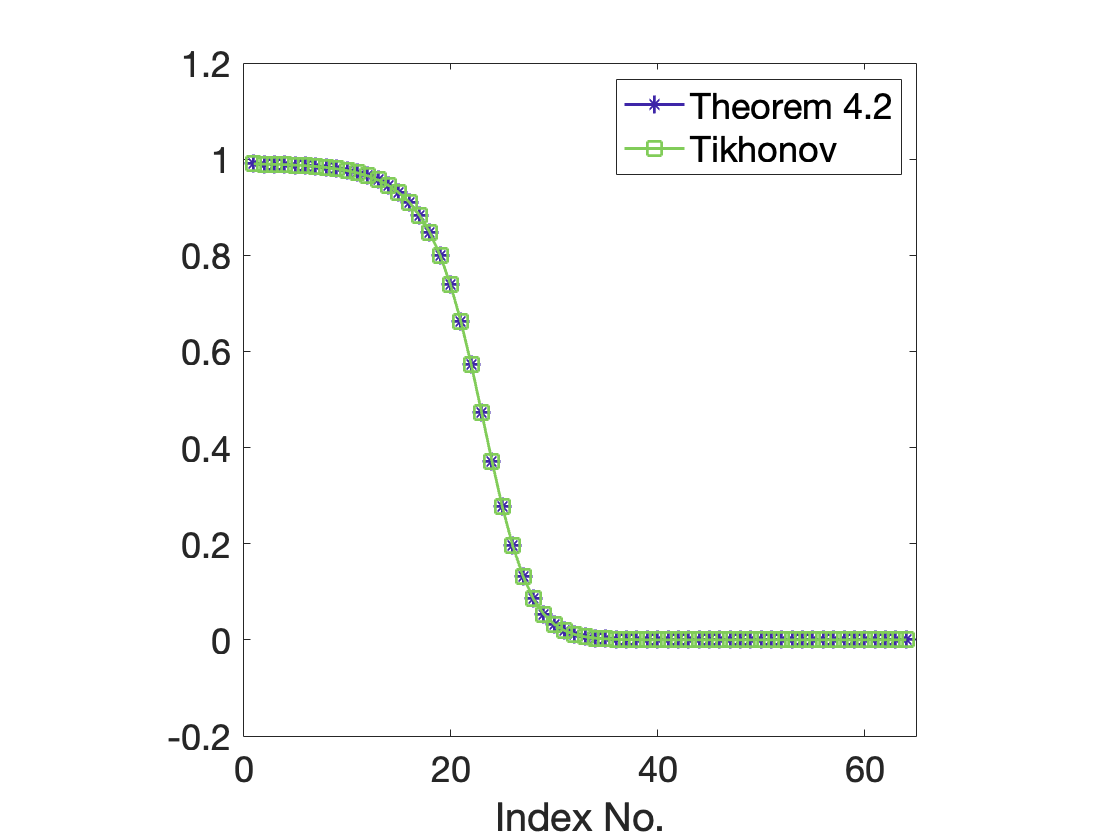}\\(b)
	\end{minipage}
	\begin{minipage}{0.32\linewidth}
		\centering
		\includegraphics[trim={6cm 0 6cm 0},clip,width=\linewidth]{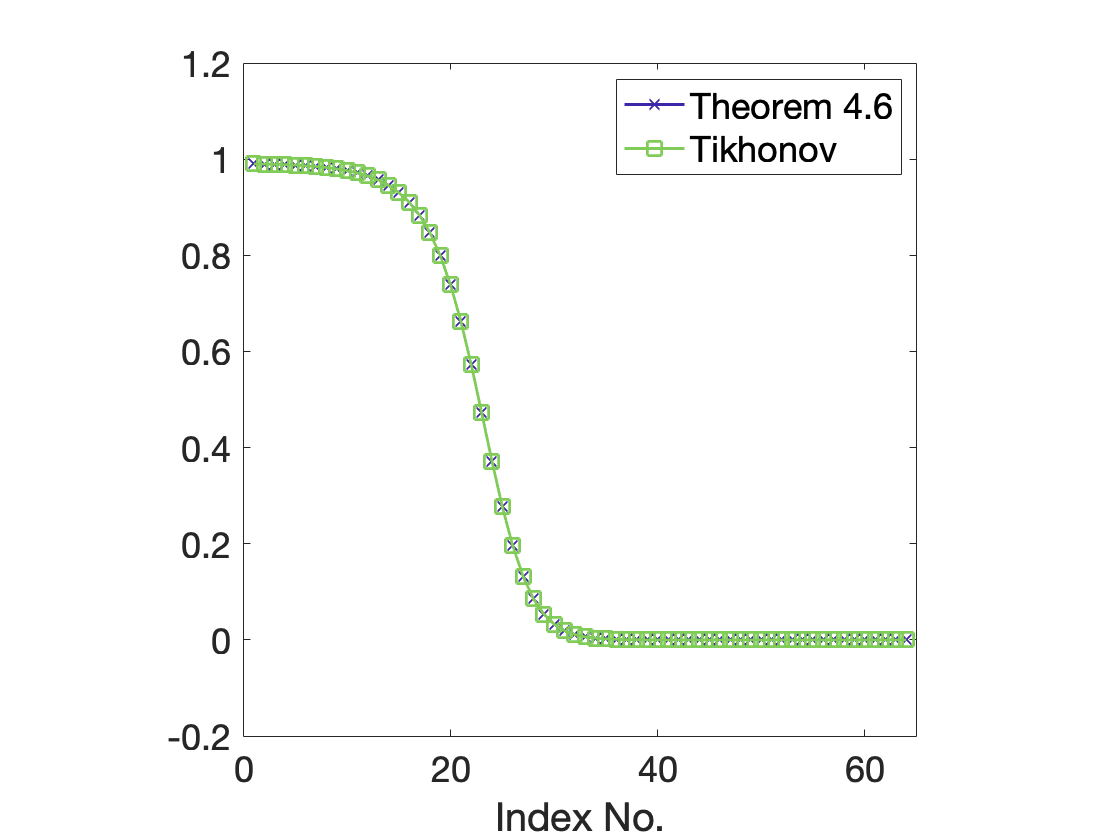}\\(c)
	\end{minipage}
\end{minipage}
\caption{\rev{\emph{Spectra} example: filter factors of standard Tikhonov \eqref{TikFF} compared against filter factors of iterative refinement computed using: (a) the effective filter factors \eqref{effectiveFF}, (b) the filter factors computed by Theorem \ref{thm1} \eqref{IR_DP_FF}, and (c) the filter factors computed by Theorem \ref{thm2} \eqref{IR_MP_FF}. The filter factors are compared using $\alpha^2 = 1e\text{-}2$ after $1$ iteration with $1\%$ noise.}}
\label{resultsFF_DP}
\end{figure}

(\emph{\textbf{Spectra - filter factors in mixed precision}}) - We utilize Theorem \ref{thm2} to investigate the filter factors of MP-IR and compare them against the effective filter factors that come from the iterates computed by Algorithm \ref{alg_3lvlIR} where again we utilize a noise level of $1\%$ and $\alpha^2=1e\text{-}2$. In Table \ref{table_FFs} we provide the summary statistics of the absolute difference between $\Phi^{(k)}$ and $\Omega^{(k)}$ of the 64 entries computed by \eqref{effectiveFF}. We note that for most precision combinations investigated, that by the $5^{th}$ iteration the mean error had converged \rev{to the unit round-off of the precision associated with Pr$_1$ (see Table \ref{choptable})}. 

We illustrate in Figure \ref{specificDiff_fig} that compared to the behavior summarized in Figure \ref{resultsFF_DP} of filter factors computed in fp64 that if the \rev{precision combination} is changed to $(\text{Pr}_1,\text{Pr}_2,\text{Pr}_3)=(3,2,1)$ that the variation of the mean entry-wise value of $|\Phi^{(k)} - \Omega^{(k)}|$ takes on approximately the unit round-off of the precision associated with the preconditioner (i.e., Pr$_1$). This can be observed in the variation of sequential entries in the top row of Figure \ref{specificDiff_fig}. Furthermore, we observe in the bottom row of the figure that the absolute largest deviation between the effective filter factors and those computed by Theorem \ref{thm2} arise in the first iteration for index numbers corresponding to a rapidly changing derivative. We surmise these differences may be a result of the orthogonality assumption made throughout or of the choice of the regularization parameter. \Rev{We close by commenting that a finite precision error analysis is expected to support these results and to help explain the impact of the preconditioner on the error.}

\begin{figure}[ht]
\centering
\begin{minipage}{1\linewidth}
	\centering
	\begin{minipage}{0.32\linewidth}
		\centering
		\includegraphics[trim={1.5cm 0 3cm 0},clip,width=\linewidth]{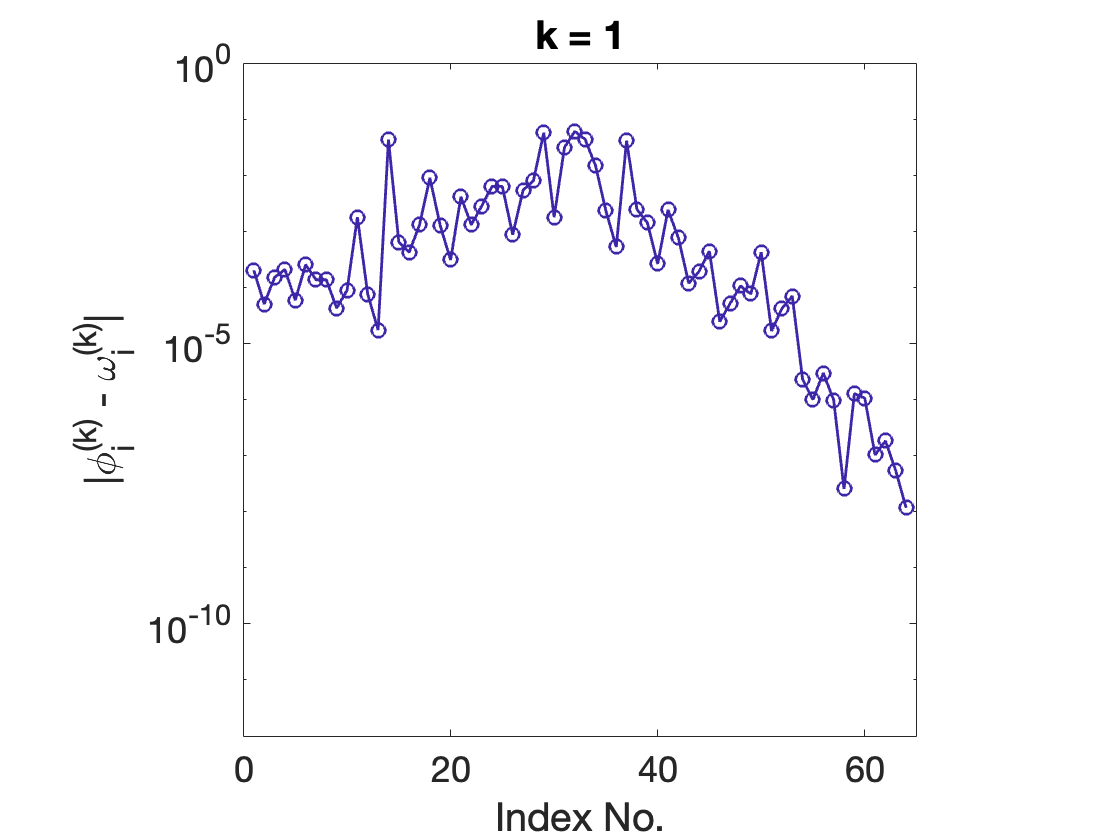}
	\end{minipage}
        \begin{minipage}{0.32\linewidth}
		\centering
		\includegraphics[trim={1.5cm 0 3cm 0},clip,width=\linewidth]{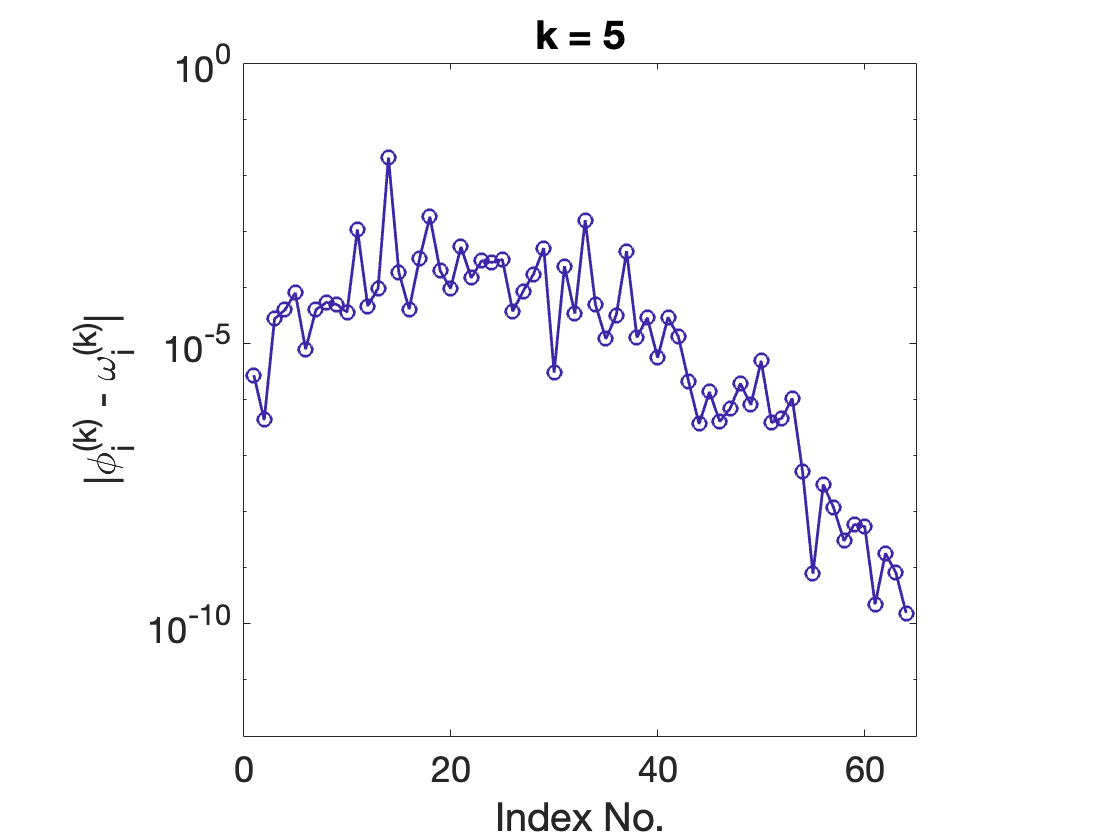}
	\end{minipage}
        \begin{minipage}{0.32\linewidth}
		\centering
		\includegraphics[trim={1.5cm 0 3cm 0},clip,width=\linewidth]{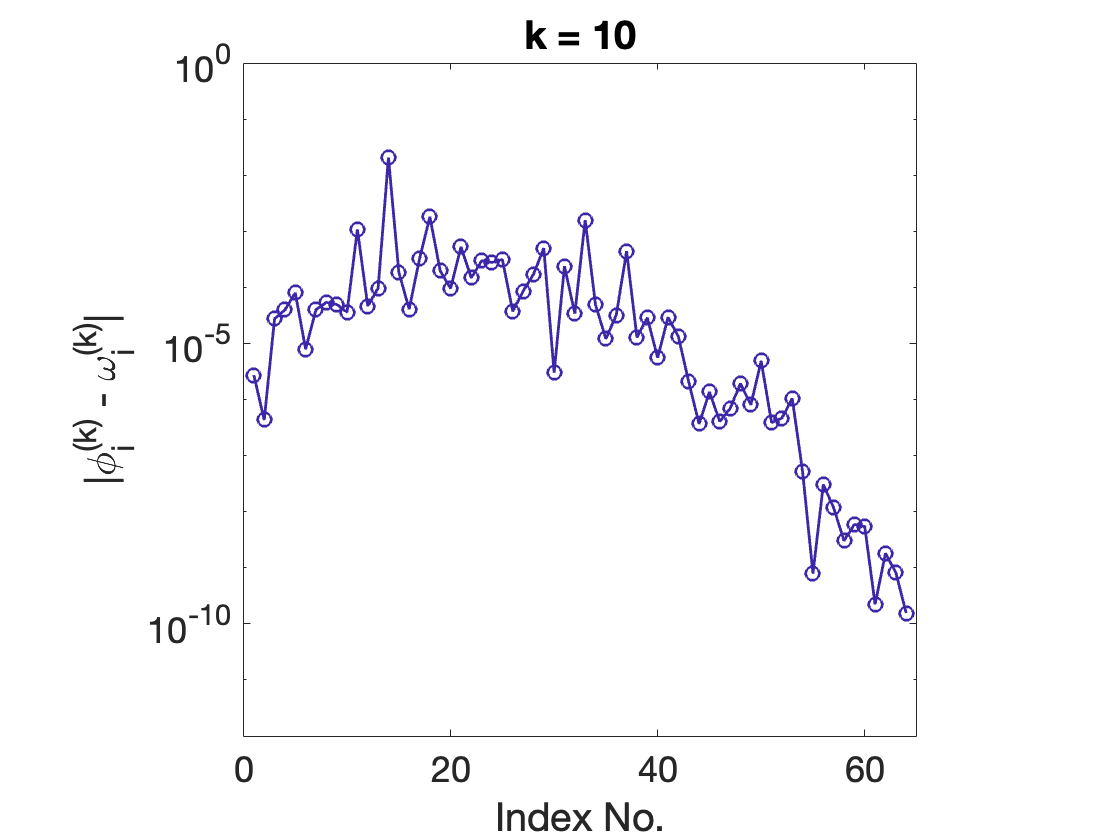}
	\end{minipage}
\begin{minipage}{0.0125\linewidth}
\centering
\phantom{\,}
\end{minipage}
        \begin{minipage}{0.31\linewidth}
		\centering
		\includegraphics[trim={5cm 0 5cm 0},clip,width=\linewidth]{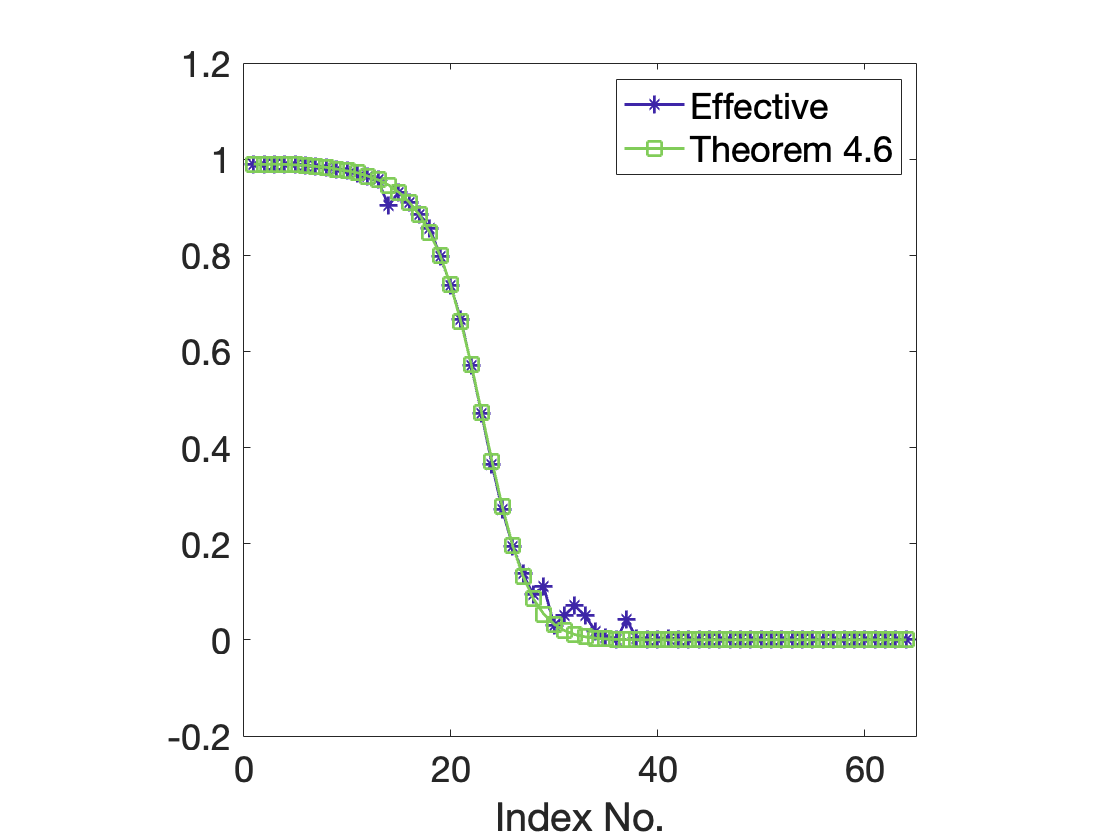}\\(a)
	\end{minipage}
 \begin{minipage}{0.005\linewidth}
\centering
\phantom{\,}
\end{minipage}
        \begin{minipage}{0.31\linewidth}
		\centering
		\includegraphics[trim={5cm 0 5cm 0},clip,width=\linewidth]{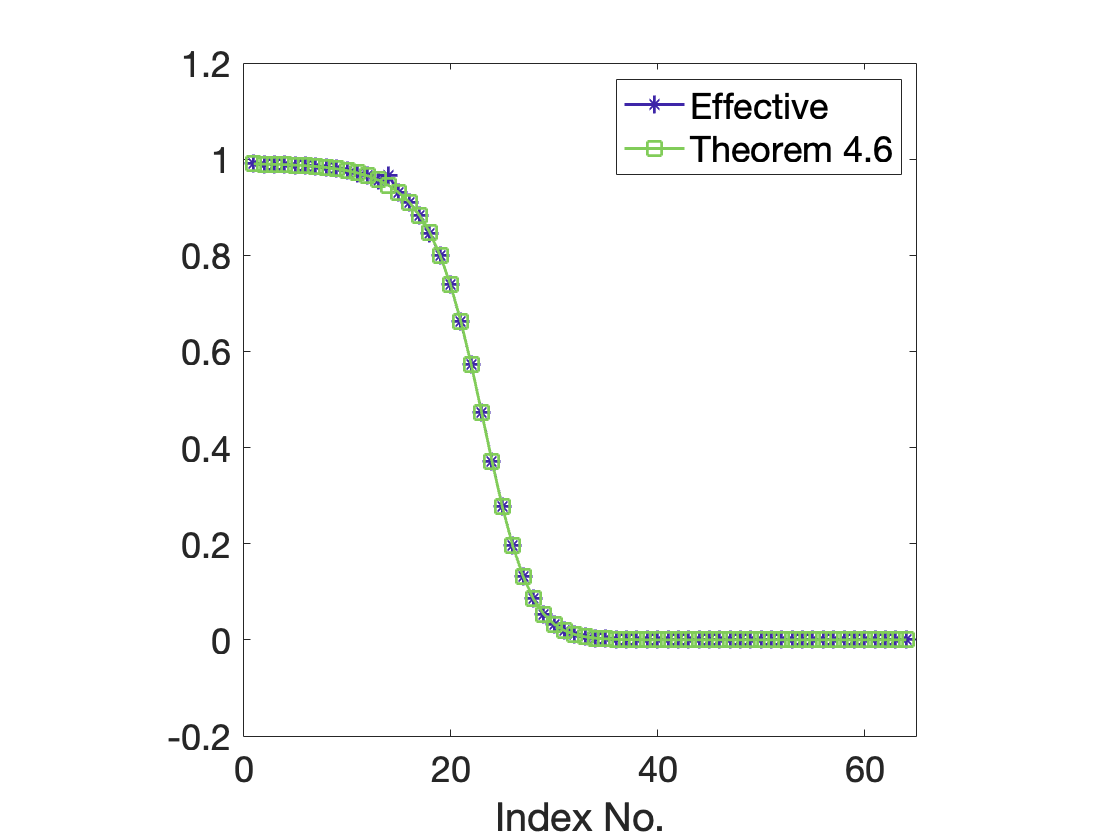}\\(b)
	\end{minipage}
 \begin{minipage}{0.005\linewidth}
\centering
\phantom{\,}
\end{minipage}
        \begin{minipage}{0.31\linewidth}
		\centering
		\includegraphics[trim={5cm 0 5cm 0},clip,width=\linewidth]{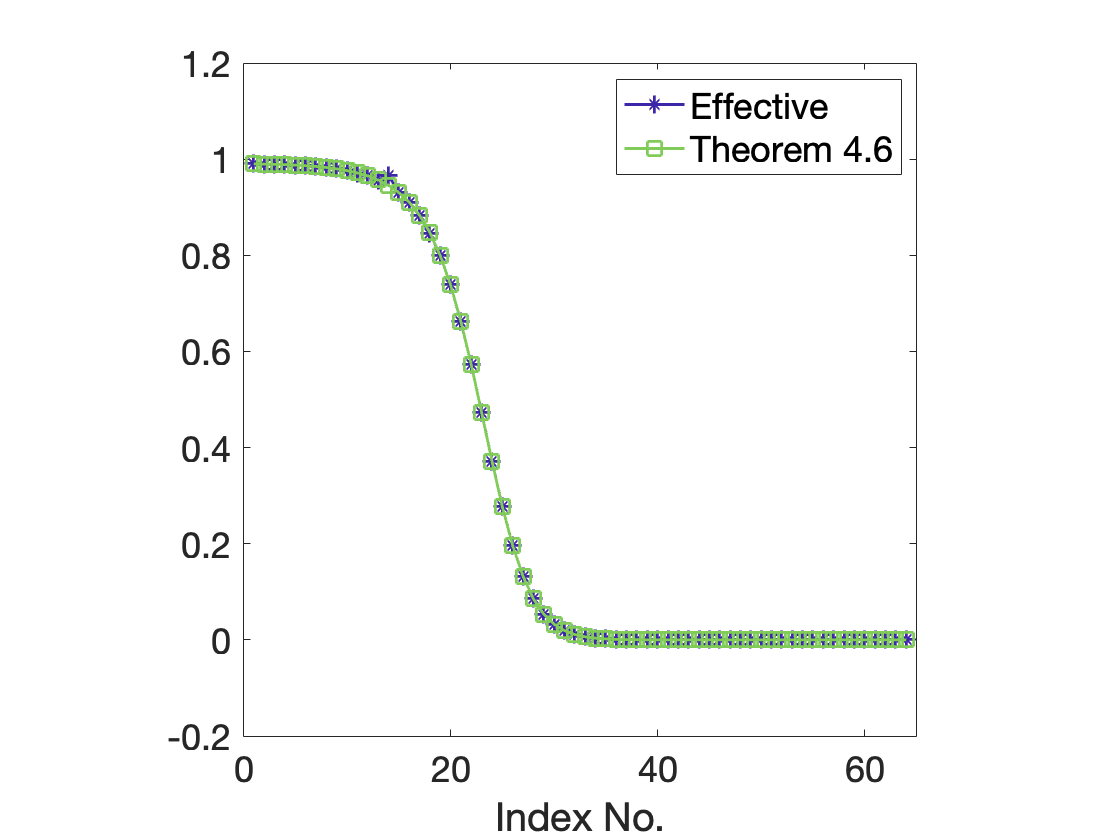}\\(c)
	\end{minipage}
\end{minipage}
\caption{\rev{MP-IR filter factors (bottom row) determined by Theorem \ref{thm2} with $(\text{Pr}_1,\text{Pr}_2,\text{Pr}_3)=(3,2,1)$ for $1\%$ noise compared to their entry-wise computations of $|\Phi^{(k)} - \Omega^{(k)}|$ (top row) at iteration numbers (a) 1, (b) 5, and (c) 10.}}
\label{specificDiff_fig}
\end{figure}

\begin{table}[ht]
        \footnotesize
        \caption{\rev{\emph{Spectra} example: summary statistics of the absolute difference between the effective filter factors and those computed by Theorem \ref{thm2} given by $|\Phi^{(k)} - \Omega^{(k)}|$ for various precision combinations with a noise level of $1\%$ and $\alpha^2=$1e-2.}}
	\centering
    \begin{tabular}{ c | c | c | c | c | c }
             \multirow{2}{*}{\scriptsize{(Pr$_1$, Pr$_2$, Pr$_3$)}} & \multirow{2}{*}{\footnotesize{Iter.}} & \multicolumn{4}{c}{\footnotesize{Summary Stats. of $|\Phi^{(k)} - \Omega^{(k)}|$}} \\
            \cline{3-6}
            & & \footnotesize{Mean Error} & \footnotesize{Min Error} & \footnotesize{Max Error} & \footnotesize{Std. Dev.} \\
        \ChangeRT{1.5pt}
        \multirow{3}{*}{\footnotesize{(1,1,1)}} & $1$& $2.0e\text{-}14$ & $1.1e\text{-}19$ & $3.8e\text{-}13$ & $6.0e\text{-}14$ \\
        \cline{3-6}
          & $5$& $4.0e\text{-}15$ & $0$ & $1.4e\text{-}13$ & $1.8e\text{-}14$ \\
        \cline{3-6}
          & $10$& $4.2e\text{-}15$ & $3.1e\text{-}21$ & $1.6e\text{-}13$ & $2.0e\text{-}14$ \\
         \hline 
          \hline
          \hline
          \multirow{3}{*}{\footnotesize{(2,1,1)}} & $1$& $2.1e\text{-}5$ & $1.3e\text{-}11$ & $3.8e\text{-}4$ & $6.7e\text{-}5$ \\
        \cline{3-6}
          & $5$& $2.7e\text{-}6$ & $2.3e\text{-}14$ & $5.2e\text{-}5$ & $7.8e\text{-}6$ \\
        \cline{3-6}
          & $10$& $2.7e\text{-}6$ & $2.3e\text{-}14$ & $5.2e\text{-}5$ & $7.8e\text{-}6$ \\
         \hline 
          \hline
          \hline
        \multirow{3}{*}{\footnotesize{(2,2,1)}} & $1$& $2.1e\text{-}5$ & $1.3e\text{-}11$ & $3.8e\text{-}4$ & $6.7e\text{-}5$ \\
        \cline{3-6}
          & $5$& $2.7e\text{-}6$ & $1.3e\text{-}13$ & $5.2e\text{-}5$ & $7.8e\text{-}6$ \\
        \cline{3-6}
          & $10$& $2.7e\text{-}6$ & $1.3e\text{-}13$ & $5.2e\text{-}5$ & $7.8e\text{-}6$ \\
         \hline 
          \hline
          \hline
          \multirow{3}{*}{\footnotesize{(2,2,2)}} & $1$& $2.0e\text{-}5$ & $9.8e\text{-}12$ & $3.8e\text{-}4$ & $6.4e\text{-}5$ \\
        \cline{3-6}
          & $5$& $2.6e\text{-}6$ & $1.6e\text{-}13$ & $5.2e\text{-}5$ & $7.8e\text{-}6$ \\
        \cline{3-6}
          & $10$& $2.8e\text{-}6$ & $6.6e\text{-}14$ & $5.3e\text{-}5$ & $8.0e\text{-}6$ \\
         \hline 
          \hline
          \hline
          \multirow{3}{*}{\footnotesize{(3,2,1)}} & $1$& $5.6e\text{-}3$ & $1.2e\text{-}8$ & $6.0e\text{-}2$ & $1.4e\text{-}2$ \\
        \cline{3-6}
          & $5$& $4.6e\text{-}4$ & $1.6e\text{-}10$ & $2.1e\text{-}2$ & $2.6e\text{-}3$ \\
        \cline{3-6}
          & $10$& $4.6e\text{-}4$ & $1.6e\text{-}10$ & $2.1e\text{-}2$ & $2.6e\text{-}3$ \\
         \hline 
          \hline
          \hline
          \multirow{3}{*}{\footnotesize{(3,2,2)}} & $1$& $5.6e\text{-}3$ & $1.2e\text{-}8$ & $6.0e\text{-}2$ & $1.4e\text{-}2$ \\
        \cline{3-6}
          & $5$& $4.6e\text{-}4$ & $1.6e\text{-}10$ & $2.1e\text{-}2$ & $2.6e\text{-}3$ \\
        \cline{3-6}
          & $10$& $4.6e\text{-}4$ & $1.6e\text{-}10$ & $2.1e\text{-}2$ & $2.6e\text{-}3$ \\
         \hline 
          \hline
          \hline
          \multirow{3}{*}{\footnotesize{(3,3,2)}} & $1$& $5.1e\text{-}3$ & $1.2e\text{-}8$ & $6.0e\text{-}2$ & $1.3e\text{-}2$ \\
        \cline{3-6}
          & $5$& $6.0e\text{-}4$ & $1.3e\text{-}11$ & $2.3e\text{-}2$ & $2.9e\text{-}3$ \\
        \cline{3-6}
          & $10$& $6.0e\text{-}4$ & $1.4e\text{-}11$ & $2.3e\text{-}2$ & $2.9e\text{-}3$ \\
         \hline 
          \hline
          \hline
          \multirow{3}{*}{\footnotesize{(3,3,1)}} & $1$& $5.1e\text{-}3$ & $1.2e\text{-}8$ & $6.0e\text{-}2$ & $1.3e\text{-}2$ \\
        \cline{3-6}
          & $5$& $6.0e\text{-}4$ & $1.4e\text{-}11$ & $2.3e\text{-}2$ & $2.9e\text{-}3$ \\
        \cline{3-6}
          & $10$& $6.0e\text{-}4$ & $1.4e\text{-}11$ & $2.3e\text{-}2$ & $2.9e\text{-}3$ \\
         \hline 
          \hline
          \hline
          \multirow{3}{*}{\footnotesize{(3,3,3)}} & $1$& $1.9e\text{-}2$ & $1.0e\text{-}8$ & $4.3e\text{-}1$ & $6.7e\text{-}2$ \\
        \cline{3-6}
          & $5$& $1.6e\text{-}3$ & $2.1e\text{-}10$ & $2.4e\text{-}2$ & $4.5e\text{-}3$ \\
        \cline{3-6}
          & $10$& $9.5e\text{-}4$ & $7.5e\text{-}10$ & $1.6e\text{-}2$ & $2.5e\text{-}3$ \\
         \hline 
		\end{tabular}
\label{table_FFs}
\end{table}

(\textbf{\emph{Spectra - reconstruction}}) - We now compare MP-IR versus AIR on the \emph{Spectra} problem. As our interests have focused on the filtering properties of MP-IR on the Tikhonov problem, we did not investigate methodologies for determining regularization parameters, nor did we consider any stopping criterion. Both topics deserve their own dedicated studies in the context of mixed precision computing for inverse problems. In our numerical investigation we chose 3 magnitudes of regularization parameters: 1e-2, 1e-3, and 1e-4 as well as two noise levels: $0.5\%$ and $3\%$ to consider the broad behavior of both methods. For method comparison, we \rev{utilize the} relative reconstructive error (RRE) defined by
\begin{equation*} \label{RRE}
    \text{RRE}\left(x^{(k)}\right) = \frac{\left\|x^{(k)}-x\right\|}{\|x\|}
\end{equation*}
where $x$ and $x^{(k)}$ denote the true solution and $k^{th}$ approximate solution determined by MP-IR or AIR, respectively. Using the RRE, we focus on the stability of the solutions by considering what we will refer to as the \emph{stable RRE} (sRRE) which is computed as the mean of the 3rd through the 10th iterates, whose iterate numbers were chosen to reflect relatively stable solution behavior. We also consider the standard deviation values for the same choice of iterates. In our numerical studies, the variants of MP-IR did not require more than 10 iterations to converge to an RRE value that would deviate more than a standard deviation from the sRRE.

\rev{Table \ref{spectraRRE_table} provides the sRRE and standard deviation values at various precisions for MP-IR as well as AIR as a function of the regularization parameter and the noise level of the problem. For succinctness in the table, we did not provide the results for $\alpha^2=1e\text{-}2$ as the presented choices resulted in smaller sRRE values for MP-IR when considered for both noise levels. When AIR was investigated closely, we found that it was stable over the considered 10 iterations with a sRRE of $6.26e\text{-}1$ and standard deviation of $4.0e\text{-}5$ for a choice of $\alpha^2=1e\text{-}1$ and $3\%$ noise, however, this is a significantly worse error than was attainable by the MP-IR method shown in the table indicating a superior MP-IR performance for these choices of $\alpha^2$.}

\rev{We observed that the sRRE values for precision combinations involving fp16 for the preconditioner (i.e., $\text{Pr}_1=3$) could experience larger sRRE variation - this is shown graphically in Figure \ref{SpectraGraphs} as well as in the standard deviation columns of Table \ref{spectraRRE_table}. Specifically, we note that while the sRRE values for all precision combinations were fairly consistent across the same level of noise, regularization parameter, and precision combination, the standard deviation of the sRRE was found to be closer to the unit round-off of Pr$_1$. This aligns with theoretical expectations given that this is the same precision that preconditioner is computed in. This variation can be observed directly in Figure \ref{SpectraGraphs}a for combination $(3,\,3,\,3)$, and additionally for combinations $(3,\,3,\,1)$ and $(3,\,3,\,2)$ in Figure \ref{SpectraGraphs}b. From this experiment it appears that variability in RRE can be reduced by balancing the choice of the preconditioner in lower precision with higher working and residual precisions.}

\begin{table}[ht]
        \footnotesize
        \caption{\rev{\emph{Spectra} example: stable relative reconstructive error (i.e., the average RRE for iterations $3$ - $10$) and the standard deviation for MP-IR and AIR methods as a function of the regularization parameter for $0.5\%$ and $3\%$ noise. MP-IR errors are provided for various mixed precision combinations.}}
	\centering
		\begin{tabular}{c | c | c | c | c | c | c }
                \multirow{3}{*}{$\alpha^2$} & \multirow{3}{*}{\footnotesize{Method}} & \multirow{3}{*}{\scriptsize{(Pr$_1$, Pr$_2$, Pr$_3$)}} & \multicolumn{4}{c}{\footnotesize{Noise Level}} \\
                 \cline{4-7}
 			& & & \multicolumn{2}{c}{$0.5\%$} & \multicolumn{2}{c}{$3\%$} \\
                \cline{4-7}
                & & & \footnotesize{sRRE} & \footnotesize{Std. sRRE} & \footnotesize{sRRE} & \footnotesize{Std. sRRE} \\
			\ChangeRT{1.5pt}
                \multirow{11}{*}{\footnotesize{1e-3}} & \footnotesize{AIR} & \footnotesize{--} & $2.0e11$ & $5.8e11$ & $1.1e12$ & $3.4e12$ \\ 
                \cline{2-7}
			& \multirow{10}{*}{\footnotesize{MP-IR}} & \footnotesize{(1,1,1)}& $3.446e\text{-}1$ & $9.7e\text{-}17$ & $3.632e\text{-}1$ & $1.5e\text{-}16$ \\
                \cline{3-7}
			  & &  \footnotesize{(2,1,1)}& $3.446e\text{-}1$ & $3.6e\text{-}12$ & $3.632e\text{-}1$ & $5.2e\text{-}12$ \\
                \cline{3-7}
			  & &  \footnotesize{(2,2,1)}& $3.446e\text{-}1$ & $6.1e\text{-}12$ & $3.632e\text{-}1$ & $1.7e\text{-}12$ \\
                \cline{3-7}
                & &  \footnotesize{(2,2,2)}& $3.446e\text{-}1$ & $6.8e\text{-}8$ & $3.632e\text{-}1$ & $7.4e\text{-}8$ \\
                \cline{3-7}
			  & &  \footnotesize{(3,2,1)}& $3.446e\text{-}1$ & $1.6e\text{-}6$ & $3.632e\text{-}1$ & $9.3e\text{-}7$ \\
                \cline{3-7}
			  & &  \footnotesize{(3,2,2)}& $3.446e\text{-}1$ & $1.5e\text{-}6$ & $3.632e\text{-}1$ & $9.5e\text{-}7$ \\
            \cline{3-7}
                & &  \footnotesize{(3,3,2)}& $3.445e\text{-}1$ & $1.7e\text{-}5$ & $3.633e\text{-}1$ & $7.6e\text{-}6$ \\
            \cline{3-7}
                & &  \footnotesize{(3,3,1)}& $3.445e\text{-}1$ & $1.9e\text{-}5$ & $3.633e\text{-}1$ & $9.7e\text{-}6$ \\
            \cline{3-7}
			  & &  \footnotesize{(3,3,3)}& $3.442e\text{-}1$ & $7.5e\text{-}4$ & $3.634e\text{-}1$ & $6.6e\text{-}4$ \\
			 \hline 
              \hline
              \hline
             \multirow{11}{*}{\footnotesize{1e-4}} & \footnotesize{AIR} & \footnotesize{--} & $6.2e19$ & $1.8e20$ & $3.7e20$ & $1.1e21$ \\ 
                \cline{2-7}
			& \multirow{10}{*}{\footnotesize{MP-IR}} & \footnotesize{(1,1,1)}& $2.129e\text{-}1$ & $3.4e\text{-}16$ & $3.289e\text{-}1$ & $2.7e\text{-}16$ \\
                \cline{3-7}
			  & &  \footnotesize{(2,1,1)}& $2.129e\text{-}1$ & $5.0e\text{-}11$ & $3.289e\text{-}1$ & $3.1e\text{-}11$ \\
                \cline{3-7}
			  & &  \footnotesize{(2,2,1)}& $2.129e\text{-}1$ & $6.8e\text{-}11$ & $3.289e\text{-}1$ & $1.4e\text{-}11$ \\
                \cline{3-7}
                & &  \footnotesize{(2,2,2)}& $2.129e\text{-}1$ & $1.4e\text{-}7$ & $3.289e\text{-}1$ & $1.5e\text{-}7$ \\
                \cline{3-7}
			  & &  \footnotesize{(3,2,1)}& $2.129e\text{-}1$ & $2.2e\text{-}5$ & $3.289e\text{-}1$ & $4.0e\text{-}5$ \\
                \cline{3-7}
			  & &  \footnotesize{(3,2,2)}& $2.129e\text{-}1$ & $2.3e\text{-}5$ & $3.289e\text{-}1$ & $4.0e\text{-}5$ \\
            \cline{3-7}
                & &  \footnotesize{(3,3,2)}& $2.128e\text{-}1$ & $1.3e\text{-}4$ & $3.289e\text{-}1$ & $1.1e\text{-}4$ \\
            \cline{3-7}
                & &  \footnotesize{(3,3,1)}& $2.129e\text{-}1$ & $1.5e\text{-}4$ & $3.289e\text{-}1$ & $1.9e\text{-}4$ \\
            \cline{3-7}
			  & &  \footnotesize{(3,3,3)}& $2.124e\text{-}1$ & $1.5e\text{-}3$ & $3.277e\text{-}1$ & $1.4e\text{-}3$ \\
		\end{tabular}
\label{spectraRRE_table}
\end{table}

\begin{figure}[ht]
\centering
\begin{minipage}{1\linewidth}
	\centering
	\begin{minipage}{0.475\linewidth}
		\centering
		\includegraphics[width=\linewidth]{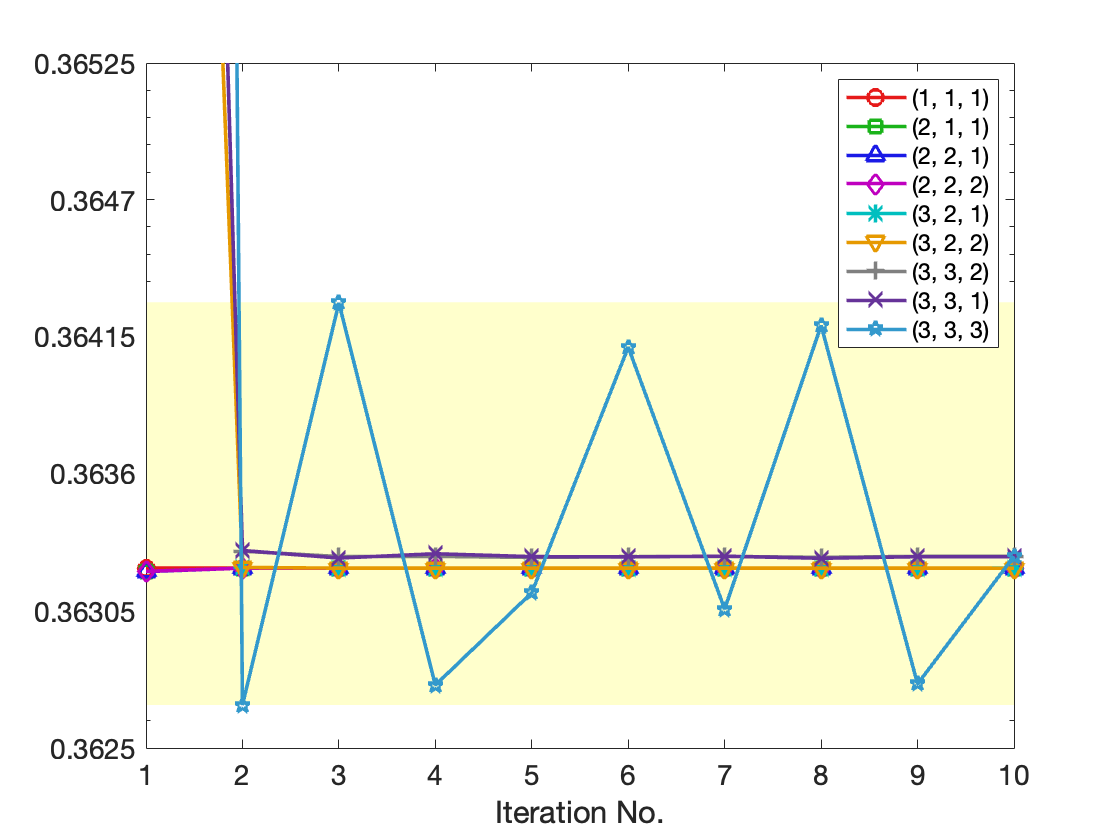}\\(a)
	\end{minipage}
        \begin{minipage}{0.475\linewidth}
		\centering
		\includegraphics[width=\linewidth]{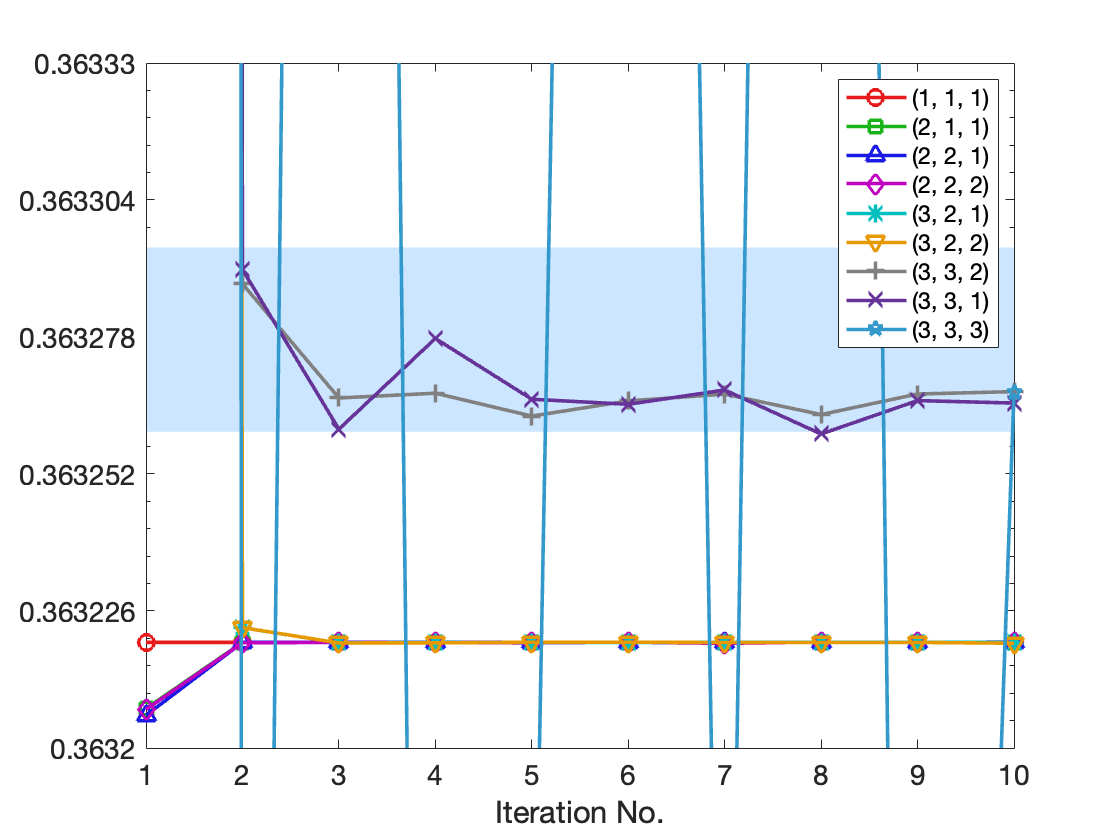}\\(b)
	\end{minipage}
\end{minipage}
\caption{\rev{\emph{Spectra} example: graph of RRE vs. iteration number for the MP-IR results for various precision combinations for $3\%$ noise and $\alpha^2=1e\text{-}3$. Pane (a) focuses on precision combination $(3,\,3,\,3)$ capturing the large variation in the yellow region compared to the other precision combinations. Pane (b) focuses on the variability of precision combinations $(3,\,3,\,1)$ and $(3,\,3,\,2)$, denoted in the blue region which is smaller than the yellow region in pane (a).}}
\label{SpectraGraphs}
\end{figure}

(\emph{\textbf{Image Deblurring - reconstruction}}) - For our final example we consider the 2D image reconstruction problem which comes from the software package IR Tools \cite{GHN18}. Similarly to the reconstruction of \emph{Spectra}, we compare the behavior of AIR against MP-IR for various precision combinations. The Hubble image for this problem in Figure \ref{gaussProbSetup}a contains $256 \times 256$ pixels which corresponds to a blurring operator of \rev{size $256^2 \times 256^2$ - making both the application of \textbf{chop} and the storage of the full problem untenable.} However, because the point spread function (Figure \ref{gaussProbSetup}b) is rank-1, the blurring matrix $A$ can be decomposed into a Kronecker product, each part of which is of reasonable size to compute its SVD \rev{in fp64, fp32, and fp16}:
\begin{equation*}
            A = A_r \otimes A_c = (U_r\Sigma_rV_r^T) \otimes (U_c\Sigma_cV_c^T)= (U_r \otimes U_c)(\Sigma_r \otimes \Sigma_c)(V_r \otimes V_c)^T.
\end{equation*}

With this Kronecker structure, it is possible to implement MP-IR efficiently; see \cite{Nagy23} for more implementation details regarding \textbf{chop} and the Kronecker product. To efficiently implement AIR one may form a normal matrix $C$ that admits a spectral factorization by imposing periodic boundary conditions when convolving with the point spread function; see \cite{HNO} for an overview of 2D image convolution. Because of the astronomical nature of the Hubble image, the choice of periodic or zero boundary conditions is typically experimentally equivalent. Zero boundary conditions were used in the formation of $A$. Figure \ref{gaussProbSetup}c shows the blurred image with $1\%$ noise.

\begin{figure}[ht]
\centering
\begin{minipage}{1\linewidth}
	\centering
	\begin{minipage}{0.32\linewidth}
		\centering
		\includegraphics[width=\linewidth]{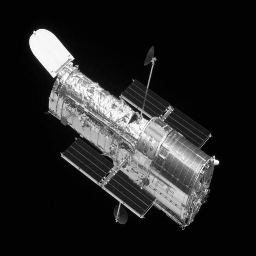}\\(a)
	\end{minipage}
        \begin{minipage}{0.32\linewidth}
		\centering
		\includegraphics[width=\linewidth]{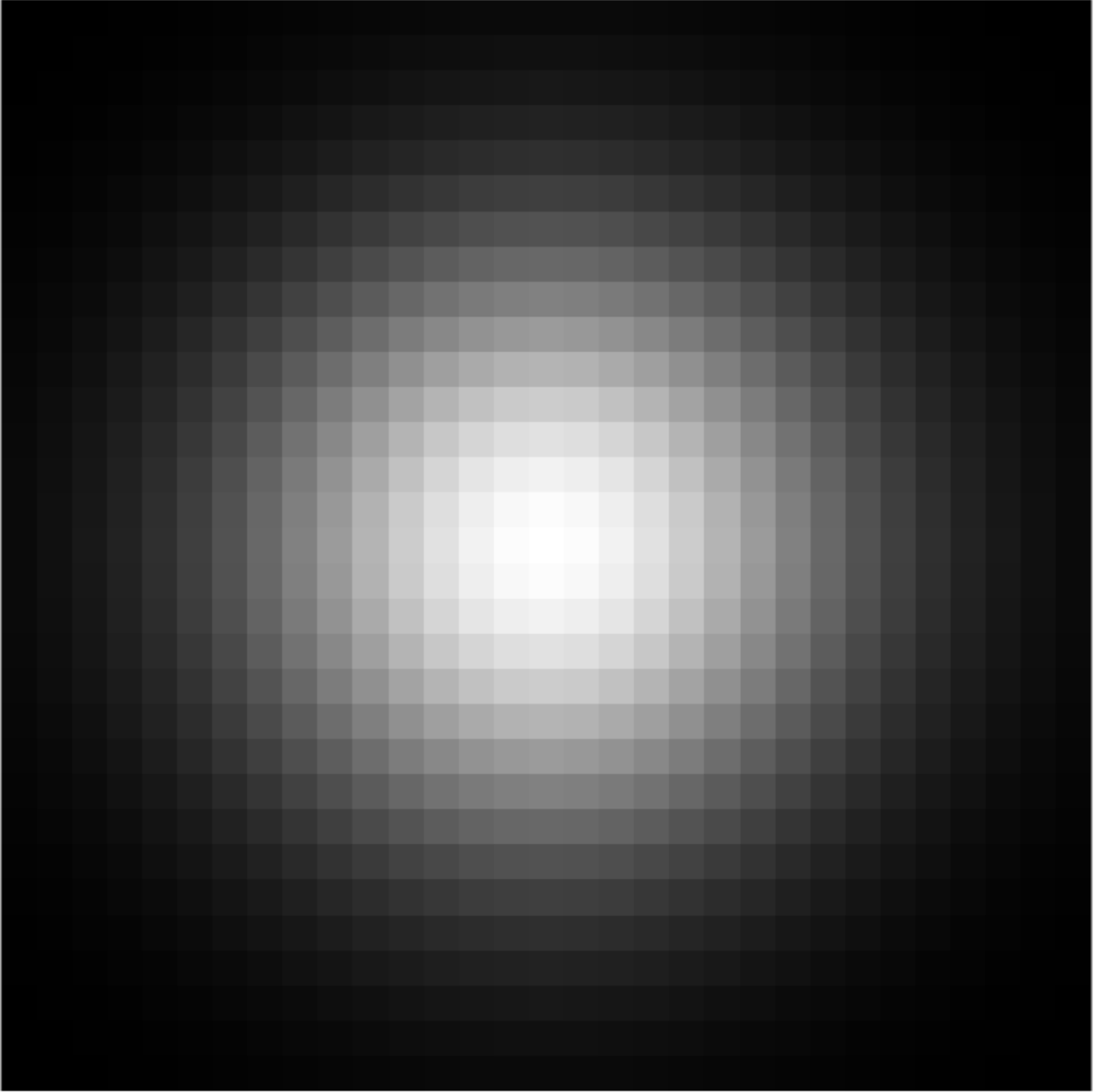}\\(b)
	\end{minipage}
        \begin{minipage}{0.32\linewidth}
		\centering
		\includegraphics[width=\linewidth]{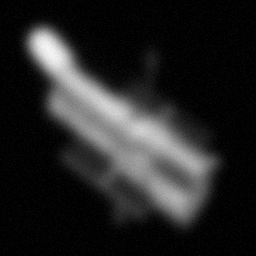}\\(c)
	\end{minipage}
\end{minipage}
\caption{\rev{Image deblurring example: (a) true Hubble image ($256 \times 256$ pixels), (b) PSF ($31 \times 31$ pixels), (c) blurred and $1\%$ noised image ($256 \times 256$ pixels).}}
\label{gaussProbSetup}
\end{figure}

Similarly to the numerical results shown for the \emph{Spectra} reconstruction example, Table \ref{table_ex2} provides \rev{the sRRE and standard deviation values at various precision combinations for MP-IR as well as AIR as a function of the regularization parameter and the noise level of the problem. Here, we note that we considered different noise levels: $1\%$ and $5\%$ and regularization parameters: $1e\text{-}1$, $1e\text{-}2$, and $1e\text{-}3$ compared to the \emph{Spectra} example, though again for the same reasons as before, we only show results for the two smaller regularization parameters.} The same general behaviors observed for the \emph{Spectra} example were also observed here \rev{with respect to observing larger sRRE standard deviations when utilizing a lower precision preconditioner.} 

We do note the observed rise in the sRRE value for the precision combination $(3,\,3,\,3)$ involving Pr$_3=3$ when the smaller $\alpha^2 = 1e\text{-}3$ was used. We attribute this behavior, in part, to the large variability experienced when using fp16 in the preconditioner (or equivalently, the solution basis) which can be seen through the standard deviation of the sRRE values in Table \ref{table_ex2}. From observing the effect that the choice of regularization parameter $1e\text{-}2$ versus $1e\text{-}3$ can make, it appears that this effect can be reduced with more regularization, but with an increase in the sRRE. For conciseness, we do not show reconstructed images using AIR and MP-IR since the reconstructions amongst the latter for the various precision combinations are indiscernible and the former is meaningless given its RRE behavior.

\begin{table}[ht]
        \footnotesize
        \caption{\rev{\emph{Image deblurring} example: stable relative reconstructive error (i.e., the average RRE for iterations $3$ - $10$) and the standard deviation for MP-IR and AIR methods as a function of the regularization parameter for $1\%$ and $5\%$ noise. MP-IR errors are provided for various mixed precision combinations.}}
	\centering
		\begin{tabular}{c | c | c | c | c | c | c }
                 \multirow{3}{*}{$\alpha^2$} & \multirow{3}{*}{\footnotesize{Method}} & \multirow{3}{*}{\scriptsize{(Pr$_1$, Pr$_2$, Pr$_3$)}} & \multicolumn{4}{c}{\footnotesize{Noise Level}} \\
                 \cline{4-7}
 			& & & \multicolumn{2}{c}{$1\%$} & \multicolumn{2}{c}{$5\%$} \\
                \cline{4-7}
                & & & \footnotesize{sRRE} & \footnotesize{Std. sRRE} & \footnotesize{sRRE} & \footnotesize{Std. sRRE} \\
			\ChangeRT{1.5pt}
                \multirow{11}{*}{\footnotesize{1e-2}} & \footnotesize{AIR} & \footnotesize{--} & $1.9e3$ & $4.6e3$ & $9.2e3$ & $2.3e4$ \\ 
                \cline{2-7}
			& \multirow{10}{*}{\footnotesize{MP-IR}} & \footnotesize{(1,1,1)}& $2.620e\text{-}1$ & $4.5e\text{-}17$ & $2.628e\text{-}1$ & $4.5e\text{-}17$ \\
                \cline{3-7}
			  & &  \footnotesize{(2,1,1)}& $2.620e\text{-}1$ & $7.2e\text{-}8$ & $2.628e\text{-}1$ & $5.7e\text{-}8$ \\
                \cline{3-7}
			  & &  \footnotesize{(2,2,1)}& $2.620e\text{-}1$ & $6.7e\text{-}8$ & $2.628e\text{-}1$ & $5.2e\text{-}8$ \\
                \cline{3-7}
                & &  \footnotesize{(2,2,2)}& $2.620e\text{-}1$ & $6.0e\text{-}8$ & $2.628e\text{-}1$ & $5.5e\text{-}8$ \\
                \cline{3-7}
			  & &  \footnotesize{(3,2,1)}& $2.621e\text{-}1$ & $2.7e\text{-}4$ & $2.629e\text{-}1$ & $2.7e\text{-}4$ \\
                \cline{3-7}
			  & &  \footnotesize{(3,2,2)}& $2.621e\text{-}1$ & $2.7e\text{-}4$ & $2.629e\text{-}1$ & $2.7e\text{-}4$ \\
            \cline{3-7}
                & &  \footnotesize{(3,3,2)}& $2.629e\text{-}1$ & $2.7e\text{-}3$ & $2.636e\text{-}1$ & $2.6e\text{-}3$ \\
            \cline{3-7}
                & &  \footnotesize{(3,3,1)}& $2.629e\text{-}1$ & $2.7e\text{-}3$ & $2.636e\text{-}1$ & $2.6e\text{-}3$ \\
            \cline{3-7}
			  & &  \footnotesize{(3,3,3)}& $2.630e\text{-}1$ & $3.2e\text{-}3$ & $2.639e\text{-}1$ & $3.3e\text{-}3$ \\
                \hline 
              \hline
              \hline
                \multirow{11}{*}{\footnotesize{1e-3}} & \footnotesize{AIR} & \footnotesize{--} & $3.9e11$ & $1.4e12$ & $1.8e12$ & $5.2e12$ \\ 
                \cline{2-7}
			& \multirow{10}{*}{\footnotesize{MP-IR}} & \footnotesize{(1,1,1)}& $2.468e\text{-}1$ & $3.3e\text{-}17$ & $2.561e\text{-}1$ & $2.0e\text{-}16$ \\
                \cline{3-7}
			  & &  \footnotesize{(2,1,1)}& $2.468e\text{-}1$ & $1.2e\text{-}7$ & $2.561e\text{-}1$ & $8.6e\text{-}8$ \\
                \cline{3-7}
			  & &  \footnotesize{(2,2,1)}& $2.468e\text{-}1$ & $2.9e\text{-}8$ & $2.561e\text{-}1$ & $6.1e\text{-}8$ \\
                \cline{3-7}
                & &  \footnotesize{(2,2,2)}& $2.468e\text{-}1$ & $1.0e\text{-}7$ & $2.561e\text{-}1$ & $7.7e\text{-}8$ \\
                \cline{3-7}
			  & &  \footnotesize{(3,2,1)}& $2.547e\text{-}1$ & $2.5e\text{-}2$ & $2.638e\text{-}1$ & $2.4e\text{-}2$ \\
                \cline{3-7}
			  & &  \footnotesize{(3,2,2)}& $2.547e\text{-}1$ & $2.5e\text{-}2$ & $2.638e\text{-}1$ & $2.4e\text{-}2$ \\
            \cline{3-7}
                & &  \footnotesize{(3,3,2)}& $2.925e\text{-}1$ & $1.4e\text{-}1$ & $3.031e\text{-}1$ & $1.5e\text{-}1$ \\
            \cline{3-7}
                & &  \footnotesize{(3,3,1)}& $2.925e\text{-}1$ & $1.4e\text{-}1$ & $3.031e\text{-}1$ & $1.5e\text{-}1$ \\
            \cline{3-7}
			  & &  \footnotesize{(3,3,3)}& $3.001e\text{-}1$ & $1.7e\text{-}1$ & $3.098e\text{-}1$ & $1.7e\text{-}1$ \\
              \hline
		\end{tabular}
\label{table_ex2}
\end{table}


\section{Conclusion} \label{sec6}

In this work, we have investigated IR on the Tikhonov problem in mixed precision. We have shown that the iterates of IR on the Tikhonov problem computed in mixed precision behave in a filtering manner by deriving methodology to formulate the iterates as a recursive relationship involving the iterates of preconditioned Landweber with a Tikhonov-type preconditioner and previous terms. Our numerical results suggest that simulated mixed precision IR applied to inverse problems \Rev{can, on average, give comparable accuracy to within a few decimal places compared to results computed only in double precision. Additionally, our results suggest that the variability of the computed solution when utilizing a low precision preconditioner can be influenced by the level of regularization as well as by the choice of the working and residual precisions.} We also found that the MP-IR method to be superior when compared to the described AIR method \rev{which may be appropriate} and efficient to compute with in image deblurring applications.

\section*{\rev{Acknowledgments}} \label{sec7}

\rev{The authors wish to thank the anonymous referees and the handling editor for carefully reading the manuscript and providing many useful suggestions that improved the presentation of the work. Additionally, they would like to thank Alan Edelman and Evelyne Ringoot for providing their invaluable computing expertise. L. O. would like to thank Marco Donatelli for hosting a research visit to Como in 2025 which included fruitful discussions that also helped improve the manuscript.
This work was partially supported by the National Science Foundation (NSF) under grants DMS-2038118 and DMS-2208294.}

\bibliographystyle{siam}
\bibliography{references}

\end{document}